\documentclass[12pt]{amsart} 

\usepackage[margin=3cm]{geometry}

\usepackage{amsmath, amsthm, amssymb, color,verbatim,bbm}
\usepackage{hyperref}      
\usepackage{units}
\usepackage{array}
\usepackage{blkarray}
\usepackage{graphicx}
\usepackage{cite}
\usepackage{mathrsfs}
\usepackage{appendix}
\usepackage{makecell}
\usepackage{arydshln}
\usepackage{cancel}
\usepackage{soul}
\usepackage{lineno}

\usepackage{tikz}
\usetikzlibrary{cd, shapes}
\usetikzlibrary{backgrounds}
\usetikzlibrary{positioning}

\newtheorem{theorem}{Theorem}
\numberwithin{theorem}{section}
\newtheorem{proposition}[theorem]{Proposition}

\theoremstyle{definition}

\newtheorem{definition}[theorem]{Definition}
\newtheorem{example}[theorem]{Example}
\theoremstyle{remark}
\newtheorem{remark}[theorem]{Remark}

\newcommand{\GT}{\mathrm{GT}}

\newcommand{\PP}{\mathbb{P}}

\newcommand{\RR}{\mathbb{R}}
\newcommand{\CC}{\mathbb{C}}
\newcommand{\ZZ}{\mathbb{Z}}

\newcommand{\ones}{\mathbbm{1}}
\newcommand{\conv}{\mathrm{conv}}
\newcommand{\capac}{\mathrm{cap}}
\newcommand{\supp}{\mathrm{supp}}
\newcommand{\rsp}{\mathrm{rowspan}}

\newcommand{\SL}{\operatorname{SL}}
\newcommand{\GL}{\operatorname{GL}}

\newcommand{\Mcal}{\mathcal{M}}

\newcommand{\KL}{\mathrm{KL}}
\newcommand{\T}{^\mathsf{T}}

\newcommand{\ci}{\perp\!\!\!\perp}

\newcommand{\exSymbol}{$\diamondsuit$}

\title[Toric invariant theory 
for ML estimation in log-linear models]{Toric invariant theory \\ 
for maximum likelihood estimation \\ in log-linear models}
\author[Carlos Am\'{e}ndola, Kathl\'{e}n  Kohn, Philipp Reichenbach, Anna Seigal]{{\normalsize Carlos Am\'{e}ndola, Kathl\'{e}n  Kohn, Philipp Reichenbach, Anna Seigal}}

\begin{document}

\begingroup
\let\MakeUppercase\relax 
\maketitle
\endgroup

\begin{abstract}
We establish connections between
invariant theory and maximum likelihood estimation for discrete statistical models.
We show that norm minimization over a torus orbit is equivalent to maximum likelihood estimation in log-linear models. 
We use notions of stability under a torus action to characterize the existence of the maximum likelihood estimate, and 
discuss connections to scaling algorithms.
\end{abstract}

\section{Introduction}

Fruitful, sometimes unexpected, connections between algebra and statistics are constantly being discovered in the field of algebraic statistics. In this paper we unveil a connection between toric invariant theory and maximum likelihood estimation for log-linear models. Log-linear models are widespread in statistics and play a fundamental role in categorical data analysis, with a wide range of applications \cite{bishop2007discrete}. They consist of discrete probability distributions whose coordinatewise logarithm lies in a fixed linear space and include, for example, independence models and discrete graphical models \cite{Lauritzen}. There is a long history of the study of log-linear models in statistics, with an emphasis on understanding their maximum likelihood inference \cite{MLEloglinear}. This concerns the existence of the maximum likelihood estimate (MLE), which maximizes the likelihood function given sample data, and statistical procedures for its computation.

Log-linear models play a prominent role in algebraic statistics \cite{Sullivant}. The key link to algebra is that the Zariski closure of a log-linear model is a toric variety, defined by a monomial parametrization. Toric varieties have a foundational place among the algebraic varieties studied in algebraic geometry \cite{cox2011toric}.

In our companion work~\cite{Gaussianpaper}, we establish a connection between finding the MLE and norm minimization along an orbit under a group action. We focus there on the setting of Gaussian group models, centered multivariate Gaussian models whose concentration matrices are of the form $g\T g$, where $g$ lies in a group. In this paper, we study the connection between invariant theory and maximum likelihood estimation in the setting of discrete exponential families. We find remarkable similarities and differences between the discrete and Gaussian settings.

The paper is organized as follows. We introduce maximum likelihood estimation and our toric invariant theory setting in the expository Sections \ref{sec:MLE} and \ref{sec:InvariantTheory}.
Our main results are in Section \ref{sec:loglinear}. We give a characterization of MLE existence in terms of 
the existence of a vector of minimal norm in an orbit under a torus action (see Theorem~\ref{thm:MLEpolystableTorus}), and an explicit way to compute the MLE from such a vector (see Theorem~\ref{thm:MLEviaMomentMap}).
We provide an alternative characterization in terms of null cones in Propositions~\ref{prop:intersectNullCones} and~\ref{prop:intersectIrredComp}.
We compare iterative proportional scaling (IPS), a classical method to find the MLE for log-linear models, with approaches to  norm minimization in Section~\ref{sec:scaling}. We conclude the paper with a comparison with the multivariate Gaussian setting of \cite{Gaussianpaper} in Section~\ref{sec:comparison} and outline a possible generalization for future research.

\section{Maximum likelihood estimation}
\label{sec:MLE} 

In this section we describe our statistical set-up: maximum likelihood estimation for discrete probability distributions.
A distribution on $m$ states is determined by its probability mass function $p$, where $p_j$ is the probability that the $j$th state occurs. Such a probability mass function is a point in the $(m-1)$-dimensional probability simplex: \[\Delta_{m-1} = \left\{ p \in \RR^m \mid p_j \geq 0 \text{ for all } j \text{ and }  \sum_{j=1}^m p_j = 1 \right\}.\]
A statistical model $\Mcal$ of distributions with $m$ states is a subset of $\Delta_{m-1}$. 
    
The data for a discrete distribution is
a \emph{vector of counts} $u \in \ZZ^{m}_{\geq 0}$, where the coordinate
$u_j$ is the number of times that the $j$th state occurs, and $n = u_+ :=~\sum_{j=1}^m u_j$ is the total number of observations. The corresponding empirical distribution is $\bar{u}=~\frac{1}{n}u \in \Delta_{m-1}$. 

Maximum likelihood estimation in $\Mcal$ given data $u$ finds a point in the model most likely to give rise to the observed data. 
That is, an MLE given $u$ is a maximizer $\hat{p}$ of the likelihood function over the model $\mathcal{M}$.
The likelihood function is 
\begin{equation}\label{eq:likelihoodiscrete}
    L_u(p) =  p_1^{u_1} \cdots p_m^{u_m}.
\end{equation}
For example, if the model fills $\Delta_{m-1}$, the likelihood is maximized uniquely at $\hat{p} = \bar{u}$. 

An MLE given $u$ is, equivalently,
a point in $\mathcal{M}$ that minimizes the \emph{Kullback-Leibler} (KL) divergence to the empirical distribution $\bar{u}$.
The KL divergence from $q \in \RR_{> 0}^m$ to $p \in \RR_{> 0}^m$ is
    \[\mathrm{KL}(p\|q) = \sum_{j=1}^m p_j \log \frac{p_j}{q_j}.\]
Although the KL divergence is not a metric, for $p,q \in \Delta_{m-1}$  it satisfies $\KL(p\|q)\geq 0$, and $\KL(p\|q)=0$ if and only if $p=q$.
The logarithm of the likelihood given $u$ can be written, up to additive constant, as
    \[\ell_u(p) = -n \sum_{j=1}^m \bar{u}_j \log \frac{\bar{u}_j}{p_j} = -n \, \KL(\bar{u} \|p).\]
We see that maximizing the log-likelihood is equivalent to minimizing the KL divergence to the empirical distribution.

\section{Toric invariant theory}
\label{sec:InvariantTheory}

In this section we describe the invariant theory  of a torus action that we will use.
We begin by introducing notions of stability under a group action. 
\subsection{Stability}
Invariant theory studies actions of a group $G$ and notions of stability with respect to this action. In this article we work with linear actions on a complex vector space. Such a linear action assigns to a group element $g \in G$ an invertible matrix in $\GL_m(\CC)$. The group element $g \in G$ acts on $\CC^m$ by left multiplication with the matrix. 
The action of group element $g$ on vector $v$ is denoted by $g \cdot v$. 
For a vector $v \in \CC^m$, we define the \emph{capacity} to be 
    \[\capac(v) := \inf_{g \in G} \| g \cdot v \|^2.\]
Here, and throughout the paper, $\Vert \cdot \Vert$ denotes the Euclidean norm for vectors and the Frobenius norm for matrices.
We now define the four notions of stability for such an action.

\begin{definition}
\label{def:StabilityNotions}
Let $v \in \CC^m$. We denote the orbit of $v$ by $G \cdot v$, the orbit closure with respect to the Euclidean topology by $\overline{G \cdot v}$ and the stabilizer $\{ g \in G : g\cdot v = v \}$ by $G_v$. We say $v$ is
    \begin{itemize}\itemsep 3pt
        \item[(a)] \emph{unstable}, if $0 \in \overline{G \cdot v}$, i.e. $\capac(v) = 0$.
        \item[(b)] \emph{semistable}, if $0 \notin \overline{G \cdot v}$, i.e. $\capac(v) > 0$.
        \item[(c)] \emph{polystable}, if $v \neq 0$ and $G \cdot v$ is closed.
        \item[(d)] \emph{stable}, if $v$ is polystable and $G_v$ is finite.
    \end{itemize}
The set of unstable points is called the \emph{null cone} of the group action.
\end{definition}

\subsection{A torus action}

We consider the $d$-dimensional complex torus, denoted $\GT_d(\CC)$ or $\GT_d$. We sometimes consider the real analogue, denoted $\GT_d(\RR)$. 
We consider the action of $\GT_d$
on a complex projective space $\PP^{m-1}_\CC$,
encoded by a $d \times m$ matrix of integers $A = (a_{ij})$.
The torus element $\lambda = (\lambda_1, \ldots, \lambda_d)$ acts on a point $v$ in $\PP^{m-1}_\CC$ by multiplication by the diagonal~matrix 
\begin{equation} \begin{bmatrix} \lambda_1^{a_{11}} \lambda_2^{a_{21}} \cdots \lambda_d^{a_{d1}} & & & \\ & \lambda_1^{a_{12}} \lambda_2^{a_{22}} \cdots \lambda_d^{a_{d2}} & & \\ & & \ddots & \\ & & & \lambda_1^{a_{1m}} \lambda_2^{a_{2m}} \cdots \lambda_d^{a_{dm}} \end{bmatrix} ,
\label{eq:torusd}
\end{equation} 
i.e. it acts on the coordinates of the point $v$ via
$v_j \mapsto \lambda_1^{a_{1j}}  \cdots \lambda_d^{a_{dj}} v_j.$

\begin{remark}
\label{rmk:realnullcone}
An action of a group $G$ on $\CC^m$ induces an action on the polynomial ring $\CC[x_1,\ldots,x_m]$ by
    $g \cdot f(x) := f\big( \, g^{-1} \cdot x \big),$
where $x = (x_1,\ldots,x_m)\T$. For the action of the torus $\GT_d$ given by matrix $A$, the map on indeterminates is $x_j \mapsto \lambda_1^{-a_{1j}}  \cdots \lambda_d^{-a_{dj}} x_j$.
\end{remark}

A \emph{linearization} 
of the action of $\GT_d$ on $\PP^{m-1}$
is a corresponding action on the underlying $m$-dimensional vector space $\CC^m$. It is given by a character of the torus, $b \in \ZZ^d$. For the linearization given by matrix $A \in \ZZ^{d \times m}$ and vector $b \in \ZZ^d$, the torus element $\lambda$ acts on the vector $v$ in $\CC^m$ via
\begin{equation}
\label{eq:torusAction}
    v_j \mapsto \lambda_1^{a_{1j}-b_1}  \cdots \lambda_d^{a_{dj}-b_d} v_j.
\end{equation}

\begin{remark}
The name linearization comes from the setting of an algebraic group acting on a complex variety $X$, as follows. We fix a line bundle over~$X$, i.e. a map $p : L \to X$, with certain properties, whose fibers are copies of $\CC$. Given a group action on $X$, a linearization is an action on $L$ that agrees with the original action under projection under $p$, and that is a linear action on each fiber~\cite[Chapter 7]{Dolgachev}. 
For example, the following projection map is a line bundle,
    \[p: \{ (x,v) \in \PP_\CC^{m-1} \times \CC^m \mid v \in \ell_x \} \to \PP_\CC^{m-1},\]
where the fiber over each $x \in \PP_\CC^{m-1}$ corresponds to the line $\ell_x$ in $\CC^m$ that the point represents. That way, a linearization lifts an action on $\PP_\CC^{m-1}$ to an action on $\CC^m$.
\end{remark} 

We now consider the notions of stability in Definition~\ref{def:StabilityNotions} for this torus action. 
They specialize to give polyhedral conditions. 
The convex hull of the 
columns $a_j \in \ZZ^d$ of the matrix $A$ is the polytope
    \[P(A) := \conv\{ a_1, \ldots, a_m \}  \subseteq \RR^d.\]
The set $P(A)$ consists of vectors in $\RR^d$ of the form $Au$ for some $u \in \Delta_{m-1}$. 
We define sub-polytopes that depend on an indexing set $J \subseteq [m]$
    \[P_J(A) := \conv\{ a_j : j \in J \}.\]
For $v \in \CC^m$, let $\supp(v):= \{ j : v_j \neq 0\} \subseteq [m]$. We abbreviate $P_{\supp(v)}(A)$ to $P_v(A)$, i.e. we define 
    \[ P_v(A) := {\rm conv}\{ a_j \mid j \in \supp(v) \} .\]
For a polytope $P \subseteq \RR^d$, we denote its interior by ${\rm int}(P)$ and its relative interior by ${\rm relint}(P)$.

\begin{theorem}[Hilbert-Mumford criterion for a torus]
\label{thm:HMtorus}
Let $v \in \CC^m$ and consider the action of the complex torus $\GT_d$ on $\CC^m$ given by matrix $A \in \ZZ^{d \times m}$ with linearization $b \in \ZZ^d$. We have
\[\begin{matrix}
(a) & v \text{ unstable} & \Leftrightarrow & b \notin P_v(A) \\
(b) & v \text{ semistable} & \Leftrightarrow & b \in P_v(A) \\
(c) & v \text{ polystable} &  \Leftrightarrow & b \in {\rm relint}(P_v(A)) \\
(d) & v \text{ stable} & \Leftrightarrow & b \in {\rm int}(P_v(A)) \end{matrix}
\]
\end{theorem}

We give an elementary proof of Theorem~\ref{thm:HMtorus} in Appendix~\ref{sec:appendixHM}. Alternative proofs can be found in~\cite[Theorem~9.2]{Dolgachev} or \cite[Theorem~1.5.1]{Szekelyhidi}.

\subsection{The moment map}
\label{sec:the_moment_map} 
We introduce the moment map and state the Kempf-Ness theorem, for a torus action.
As before, $\GT_d$ denotes the $d$-dimensional complex torus, and we consider its action on $\CC^m$ 
via the matrix $A \in \ZZ^{d \times m}$. We first consider the trivial linearization $b=0$ and later a general linearization $b \in \ZZ^d$.

Fix $v \in \CC^m$ and consider a torus element $\lambda = (\lambda_1, \ldots, \lambda_d)$ in $\GT_d$.
The $j$th coordinate of the vector $\lambda \cdot v \in \CC^m$ is
 \[(\lambda \cdot v)_j = \lambda_1^{a_{1j}} \cdots \lambda_d^{a_{dj}} v_j.\]
Next, we consider the map that sends $\lambda$ to the squared norm of $\lambda \cdot v$:
    \[\begin{split}
    \gamma_v : \GT_d & \to \RR  \\
    \lambda & \mapsto \| \lambda \cdot v \|^2 = \sum_{j=1}^m | \lambda_1|^{2a_{1j}} \cdots |\lambda_d|^{2a_{dj}} {|v_j|}^2 .
    \end{split}\]
The infimum of $\gamma_v$ over $\lambda \in \GT_d$ is the capacity of $v$.

More generally, for an algebraic group $G$ acting linearly on a space $V$ we can consider the map $\gamma_v : G \to \RR$ that sends $g \mapsto \| g \cdot v \|^2$, for fixed $v \in V$. 
The derivative is a map $D_I \gamma_v: T_I G \to \RR$, 
where $T_I G$ is the tangent space to $G$ at $I$.
The \emph{moment map} $\mu$ assigns to $v \in V$ the derivative of the map $\gamma_v$ at $I \in G$.

For the group $\GT_d$, the tangent space at $I$ is equal to $\CC^d$, and the derivative is a map $\CC^d \to \RR$.
Recall that the map $f: \CC \to \CC$, $z \mapsto |z|^2$ is not complex differentiable. We identify $\CC$ with $\RR^2$, writing $z = x + iy$. Then the differential of $f$ is given in terms of $x$ and $y$, and their tangent directions $\dot{x}$ and $\dot{y}$, as $\dot{z} \mapsto 2x\dot{x} + 2y\dot{y}$. In particular, the differential at $z=1$ is the map $\dot{z} \mapsto 2  \Re(\dot{z})$, where $\Re(\cdot)$ denotes the real part of a complex scalar. 
Extending to the multivariate setting, we obtain the derivative map
\[\begin{split}
D_I \gamma_v : \CC^d & \to \RR \\
\dot{\lambda} & \mapsto \sum_{i=1}^d \left( \sum_{j=1}^m 2 a_{ij}  |v_j|^2 \right) \Re(\dot{\lambda_i}) = 2 \Re \left( \sum_{i=1}^d (A v^{(2)})_i \dot{\lambda}_i \right) ,
\end{split}\]
where $v^{(2)}$ is the vector with $j$th coordinate $|v_j|^2$. 
We can identify $\CC^d$ with the space of $\RR$-linear functionals ${\rm Hom}(\CC^d,\RR)$, by associating to $u \in \CC^d$ the map $w \mapsto \Re ( \sum_{i=1}^d u_i w_i )$. Under this identification, the linear map $D_I \gamma_v$ corresponds to the vector $2 A v^{(2)} \in \CC^d$. 
Hence the moment map, for linearization $b=0$, is
    \[ \begin{matrix} \mu : & \CC^m  & \longrightarrow & \CC^d \\ 
    & v & \longmapsto & 2A v^{(2)} .\end{matrix}\]
For a general linearization $b \in \ZZ^d$, we replace the columns $a_j$ of $A$ by $a_j - b$. 
This replaces the vector $A v^{(2)}$ by $A v^{(2)} - b \| v \|^2$, where $\| v \|^2 = \sum_{j=1}^m | v_j|^2$.
We obtain
    \[ \begin{matrix} \mu : & \CC^m  & \longrightarrow & \CC^d \\ 
    & v & \longmapsto & 2(A v^{(2)}  - \| v \|^2 b ).\end{matrix}\]

The Kempf-Ness theorem relates points of minimal norm in an orbit, or orbit closure, to the vanishing of the moment map. It was first  proven in \cite{KempfNess}. Nowadays, several statements are referred to as Kempf-Ness theorem; see \cite[Section~2]{Gaussianpaper} for a summary.
For our torus action, we obtain the following. 

\begin{theorem}[Kempf-Ness theorem for a torus]
\label{thm:kempfNessTorus}
Consider the torus action of $\GT_d$ given by matrix $A \in \ZZ^{d \times m}$ with linearization $b \in \ZZ^d$. A vector is semistable (resp. polystable) if and only if there is a non-zero $v$ in its orbit closure (resp. orbit) with $A v^{(2)} = \| v \|^2 b$.
This $v$ is a vector of minimal norm in the orbit closure (resp. orbit).
\end{theorem}

We give two proofs of  Theorem~\ref{thm:kempfNessTorus} in Appendix~\ref{sec:appendixKN}. The first proof is a translation from the original paper of Kempf and Ness~\cite{KempfNess}; the second proof uses Theorem~\ref{thm:HMtorus}.

\subsection{The null cone}

The set of unstable points under a group action on a vector space is the null cone, see Definition~\ref{def:StabilityNotions}. In many settings of interest, the null cone is a Zariski closed set, the vanishing locus of all homogeneous invariants of positive degree.
It is a classical object of interest, studied by Hilbert~\cite{hilbert1890ueber}.

We recall the setting of the action of the complex torus $\GT_d$ on $\CC^m$ given by a matrix $A \in \ZZ^{d \times m}$ with linearization $b \in \ZZ^d$. The stability of $v \in \CC^m$ is determined by its support $\supp(v)$; see Theorem~\ref{thm:HMtorus}. In particular the null cone, as a set, is a union of coordinate linear spaces. 
We describe it in terms of the standard basis vectors in $\CC^m$, denoted $e_1, \ldots, e_m$.
The linear space spanned by $\{ e_j : j \in J \}$ is denoted $\langle e_j : j \in J \rangle$. 

A vector $b$ in $P(A)$ can be written as a convex combination of the $m$ columns of $A$. We consider the maximal sub-polytope of $P(A)$ that does not contain $b$, as well as the minimal sub-polytope of $P(A)$ that contains $b$. Both minimality and maximality are taken with respect to inclusion in the set $[m]$.

In Proposition~\ref{prop:null_cone_linear_spaces}, we see the connection between irreducible components of the null cone and maximal sub-polytopes of $P(A)$ not containing $b$, see Figure~\ref{fig:max_not_containing_b}. Then, in Proposition~\ref{prop:vanishing_monomials}, we see that minimal sub-polytopes containing $b$ give set-theoretic defining equations for the null cone, see Figure~\ref{fig:minimal_containing_b}.

\begin{proposition}
\label{prop:null_cone_linear_spaces}
Consider the action of $\GT_d$ on $\CC^m$ given by matrix $A \in \ZZ^{d \times m}$ with linearization $b \in \ZZ^d$. 
The irreducible components of the null cone are the linear spaces 
$\langle e_j : j \in J \rangle$, where $P_J(A)$ is a maximal sub-polytope of $P(A)$ with $b \notin P_J(A)$.
\end{proposition}

\begin{proof}
Assume that a point $v \in \CC^m$ lies in a linear space $\langle e_j : j \in J \rangle$ where $b \notin P_J(A)$. Then $\supp(v) \subseteq J$, hence $b \notin P_v(A)$, and $v$ is unstable by Theorem~\ref{thm:HMtorus}(a). Conversely, assume that $v \in \CC^m$ is not contained in any linear space $\langle e_j : j \in J \rangle$ as in the statement. Since the $P_J(A)$ are maximal with $b \notin P_J(A)$, we have $b \in P_v(A)$ and $v$ is semistable.
\end{proof}

\begin{figure}[htbp]
    \centering
    \includegraphics[width=3cm]{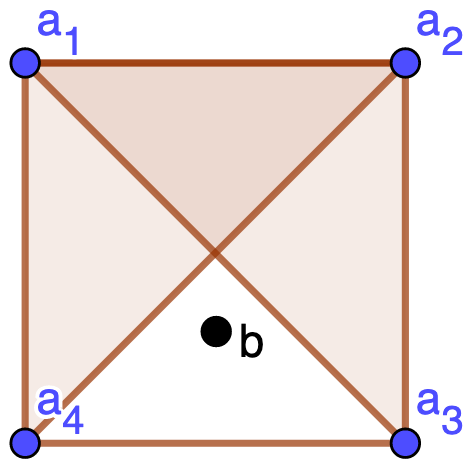} \qquad 
    \includegraphics[width=3cm]{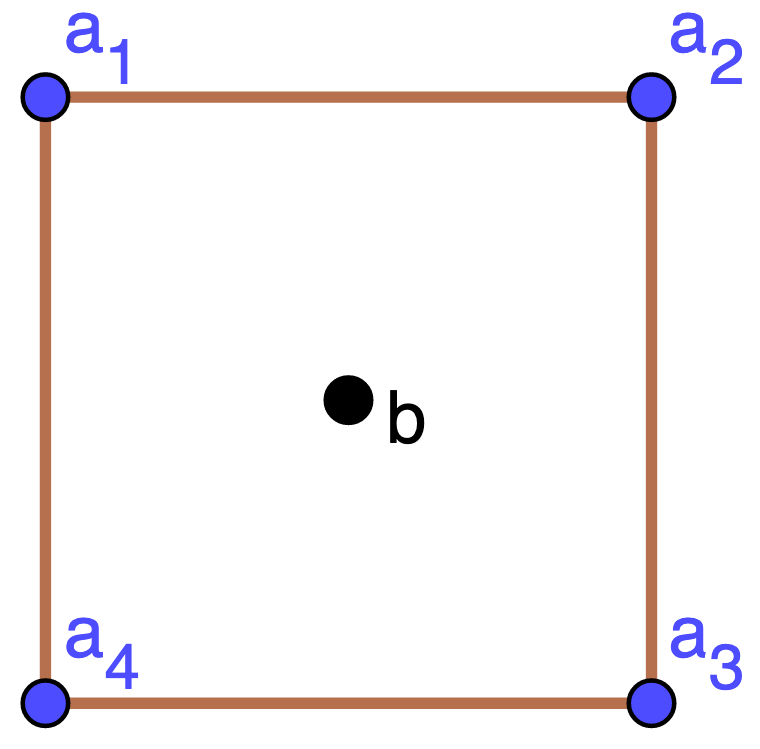}
    \qquad
    \includegraphics[width=2.9cm]{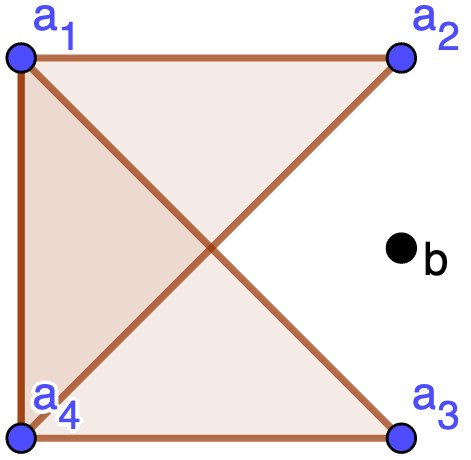}
    \qquad
    \includegraphics[width=3cm]{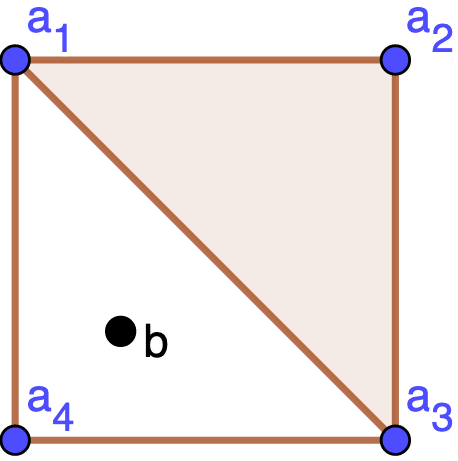}
    \caption{The maximal sub-polytopes of $P(A)$ not containing $b$, for four different $b \in \ZZ^2$.
    For example, the leftmost picture displays $P_J(A)$ for $J$ equal to $\{1,2,3\}$, $\{1, 2, 4\}$, and $\{3, 4\}$.
    Each sub-polytope corresponds to an irreducible component of the null cone, see Proposition~\ref{prop:null_cone_linear_spaces}. For $b$ on the boundary of $P(A)$, all maximal sub-polytopes intersect, see Proposition~\ref{prop:intersectIrredComp}.}
    \label{fig:max_not_containing_b}
\end{figure}

\begin{proposition} 
\label{prop:vanishing_monomials}
Consider the action of $\GT_d$ on $\CC^m$ given by matrix $A \in \ZZ^{d \times m}$ with linearization $b \in \ZZ^d$. 
A vector $v \in \CC^m$ is in the null cone if and only if all products $\prod_{j \in J} v_j$ vanish, where $J \subseteq [m]$ indexes a minimal sub-polytope of $P(A)$ containing $b$.
\end{proposition}

\begin{proof}
Denote $v_J := \prod_{j \in J} v_j$. If some $v_J$ is non-zero, i.e. $J \subseteq \supp(v)$, then $b \in P_J(A)$ implies $b \in P_v(A)$, hence $v$ is semistable by Theorem~\ref{thm:HMtorus}(b). Conversely, if $v$ is semistable then $b \in P_v(A)$. By minimality of the minimal sub-polytopes $P_J(A)$ containing $b$ we have, for some $J$ in the statement, the containment $P_J(A) \subseteq P_v(A)$, i.e. $J \subseteq \supp(v)$, hence $v_J \neq 0$. 
\end{proof}

\begin{figure}[htbp]
    \centering
    \includegraphics[width=3cm]{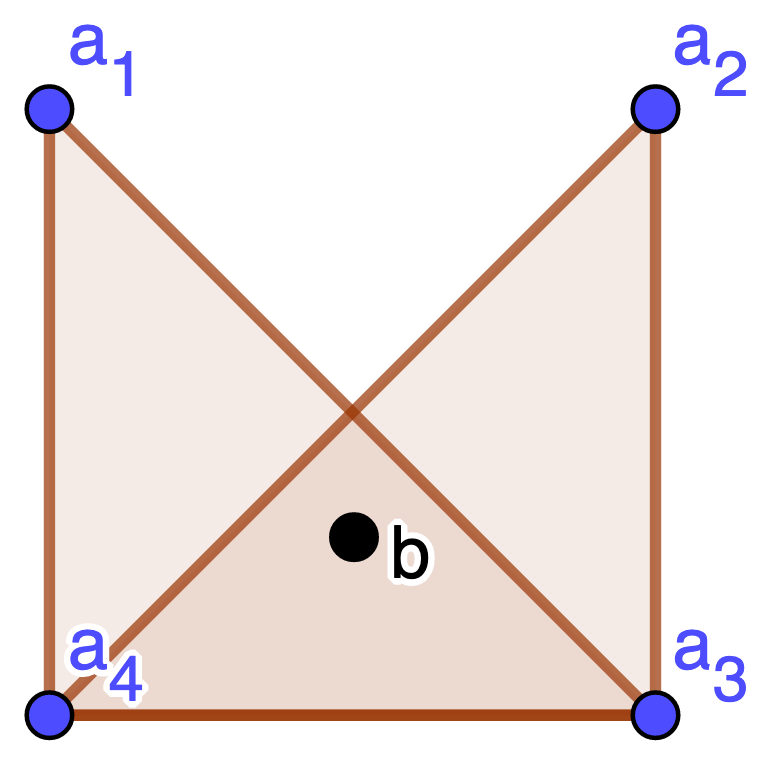} \qquad 
    \includegraphics[width=3cm]{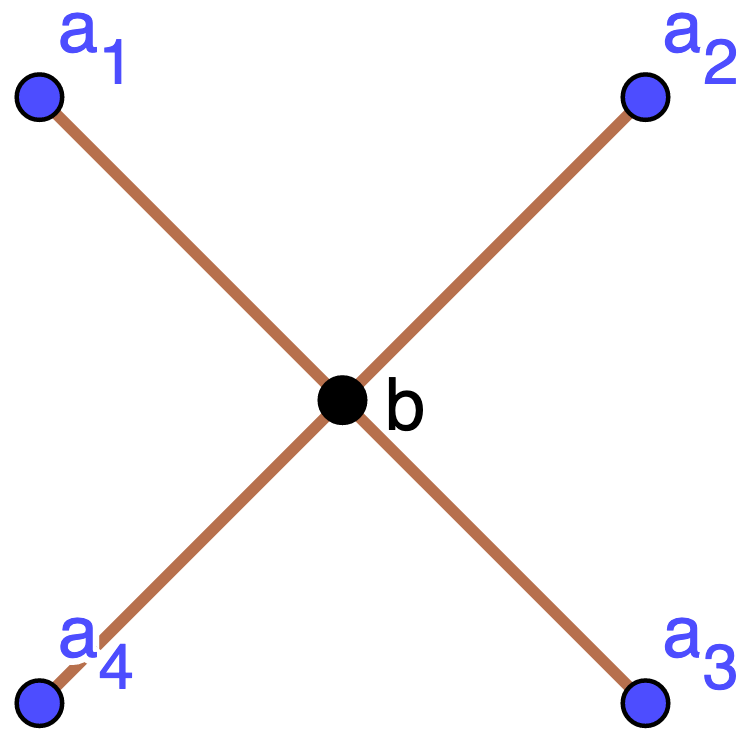}
    \qquad 
    \includegraphics[width=3cm]{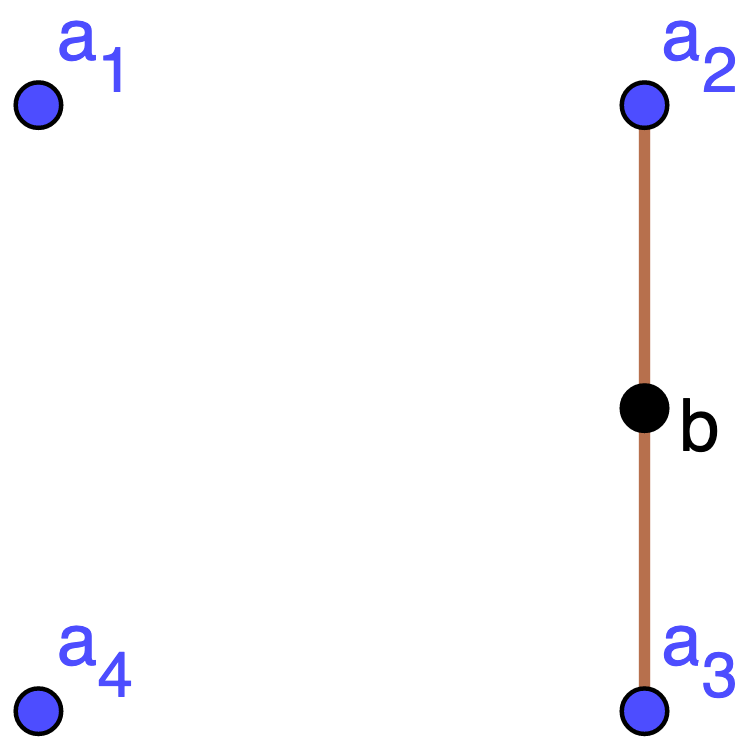}
        \qquad 
    \includegraphics[width=3cm]{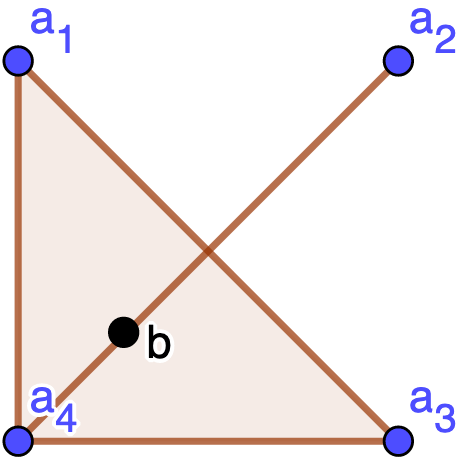}
    \caption{The minimal sub-polytopes of $P(A)$ containing $b$, for four choices of $b \in \ZZ^2$.
    For example, the leftmost picture displays $P_J(A)$ for $J$ equal to $\{1,3,4\}$ and $\{2,3, 4\}$.
    Each sub-polytope corresponds to a generator of the null cone, see Proposition~\ref{prop:vanishing_monomials}.}
    \label{fig:minimal_containing_b}
\end{figure}

The null cone is defined by the vanishing of all homogeneous invariants of positive degree. The monomials from Proposition~\ref{prop:vanishing_monomials} give the square-free part of the generators of the null cone. We describe how to take powers of the indeterminates appearing in the monomials, in order to turn them into invariants.
Let $J \subseteq [m]$ index a minimal sub-polytope of $P(A)$ containing $b$. Then $0$ can be written as a strictly positive convex combination of 
$\{ (a_j -b) : j \in J \}$.
Since the entries of the matrix $A$ and the vector $b$ are integers, the convex combination is rational. Multiplying by the lowest common denominator gives a positive integer linear combination
\begin{equation}
\label{eqn:convex_zero}
     \sum_{j \in J} r_j (a_j -b) = 0 ,  \qquad r_j \in \ZZ_{>0}.
\end{equation}
The monomials $\prod_{j \in J} x_j^{r_j}$ are invariants under the group action, since
    \[ \lambda \cdot \left( \prod_{j \in J} x_j^{r_j} \right) = \prod_{j \in J} \left( \lambda^{-(a_j - b)} x_j \right)^{r_j} = \prod_{j \in J} x_j^{r_j} \cdot \lambda^{- \sum_{j \in J} r_j (a_j - b)} =  \prod_{j \in J} x_j^{r_j} , \]
where the first equality follows from Remark~\ref{rmk:realnullcone} and the last equality follows from~\eqref{eqn:convex_zero}.

\section{Main Results}
\label{sec:loglinear}

We begin this section by introducing maximum likelihood estimation in log-linear models, pointing out connections to torus actions.
We relate stability under a torus action to maximum likelihood estimation for log-linear models in Section~\ref{sec:stability_MLE}, and describe how to compute the MLE from the moment map in Section~\ref{sec:torusmoment}.

\subsection{Log-linear models}

A log-linear model consists of distributions whose logarithms lie in a fixed linear space. 
The log-linear model corresponding to a matrix  $A \in \ZZ^{d \times m}$ is
 \begin{equation}
     \label{eq:logLinearModel}
 \mathcal{M}_A = \{ p \in \Delta_{m-1} \, \vert \, \log p \in \rsp(A) \}.
 \end{equation}
 The coordinatewise logarithm $\log p$ applies to $p$ with strictly positive entries, and we therefore have $\mathcal{M}_A \subseteq {\rm relint}(\Delta_{m-1})$.
A parametrization of the model $\mathcal{M}_A$ is given by
\begin{equation}
\label{eq:toricparam}
\begin{matrix} \phi^A: & \RR_{>0}^d & \longrightarrow & \Delta_{m-1} \\
    & \theta & \longmapsto &  \left( \frac{1}{Z(\theta)} \prod_{i=1}^d \theta_i^{a_{ij}} \right)_{1 \leq j \leq m} \end{matrix}
\end{equation}
where $Z$ is a normalization factor. We observe a first connection between the statistical model and a torus action: the map $\phi^A$ is, up to normalization, the action~\eqref{eq:torusAction} of the real positive torus element $\theta$ on the all-ones vector $\ones  = (1,\dots,1) \in \RR^m$ with trivial linearization $b = 0$. 

For the log-linear model $
\mathcal{M}_A$, we assume that the vector $\ones$ is in the row span of $A$; this is
a common assumption for statistical, as well as algebraic, reasons. First, such log-linear models are equivalent to discrete exponential families \cite[Section 6.2]{Sullivant}. Second, the assumption means the uniform distribution $\frac{1}{m}\ones$ is in the model. Moreover, consider the  Zariski closure of~$\mathcal{M}_A$ in $\CC^m$, defined by the ideal
\begin{equation}
\label{eq:ZarMA}
    I_A = \langle p^v - p^w \, \vert \, v,w \in \ZZ_{\geq 0}^m \text{ such that } Av = Aw \rangle 
\end{equation}
in the ring $\CC[p_1,\dots,p_m]$, where $p^v := \prod_{j=1}^m p_j^{v_j}$ for $v \in \ZZ_{\geq 0}^m$.
If $\ones \in \rsp(A)$, this becomes a homogeneous ideal: indeed, if $r\T A = \ones $ for some $r \in \RR^d$ then multiplying $Av = Aw$ by this vector results in $\ones v = \ones w$. 

\begin{example}
\label{ex:indep}
A probability distribution on two ternary random variables is a $3 \times 3$ matrix $p = (p_{ij})$ of non-negative entries that sum to one. A distribution lies in the independence model if
    \[ p_{ij} = p_{i+} p_{+j} \quad \text{ for all } \quad 1 \leq i,j \leq 3,\]
where $p_{i+}$ is the sum of the $i$th row of $p$, and $p_{+j}$ the sum of the $j$th column. 
We view the independence model as a discrete exponential family, hence require that the entries of $p$ are strictly positive.
The independence model on a pair of ternary random variables is the log-linear model $\Mcal_A$ where
\[ A = \begin{bmatrix} 1 & & & 1 & & & 1 & & \\ & 1 & & & 1 & & & 1 & \\ & & 1 & & & 1 & & & 1 \\ 1 & 1 & 1 & & & & & & \\ & & & 1 & 1 & 1 & & & & \\ & & & & & & 1 & 1 & 1 \\
\end{bmatrix}  \in \ZZ^{6 \times 9}. \]
A distribution $p \in \Mcal_A$ has nine states, i.e. $\Mcal_A \subseteq \Delta_8$. We identify $\RR^9$ with $\RR^{3 \times 3}$ to view the nine state random variable as a pair of ternary random variables.

The action of $\GT_6(\RR)$ on $\RR^{3 \times 3}$ given by~\eqref{eq:torusd}, is as follows. The torus element 
    \[ \left( \begin{matrix} \nu_1 & \nu_2 & \nu_3 & \nu_4 & \nu_5 & \nu_6 \end{matrix} \right) = \left( \begin{matrix} \lambda_1 & \lambda_2 & \lambda_3 & \mu_1 & \mu_2 & \mu_3 \end{matrix} \right) \]
acts on a matrix $x \in \RR^{3 \times 3}$ by multiplying the entry $x_{ij}$ by $\prod_{k = 1}^6 \nu_k^{a}$ where $a$ is the column of $A$ with  index $(i,j)$. 
This is the left-right action of $\GT_3 \times \GT_3$ on the space of $3 \times 3$ matrices; it sends $x_{ij} \mapsto \lambda_i \mu_j x_{ij}$.
In particular, the orbit of the all-ones matrix consists of all rank one matrices with all entries non-zero. The intersection of this orbit with $\Delta_8$ gives the independence model on a pair of ternary variables. \hfill\exSymbol
\end{example}

We now consider maximum likelihood estimation for log-linear models. 
Recall that an observed vector of counts $u \in \ZZ^{m}_{\geq 0}$ defines an empirical distribution~$\bar{u} \in \Delta_{m-1}$. 
The vector $A\bar{u}$ is a vector of sufficient statistics for the model $\mathcal{M}_A$.
A maximum likelihood estimate is a point $q \in \mathcal{M}_A$ such that
\begin{equation}\label{eq:Birch}
Aq = A\bar{u},
\end{equation}
see e.g. \cite[Proposition 2.1.5]{owl} or \cite[Corollary 7.3.9]{Sullivant}.

Since the model $\Mcal_A$ is not closed, the MLE may not exist. Birch \cite{birch1963} was the first to rigorously study MLE existence in the context of multi-way tables, where he observed that $u$ having all entries strictly positive is a sufficient condition for the MLE to exist and derived condition \eqref{eq:Birch}, sometimes known as Birch's Theorem, 
see~\cite[Theorem 1.10]{ASCB}. The fact that some entries could still be $0$ without affecting MLE existence was not fully understood until the work of Haberman, who gave the first characterization of MLE existence in her paper~\cite{haberman1974}. A modern necessary and sufficient condition is the following.

\begin{proposition}[\!\!{\cite[Theorem 8.2.1]{Sullivant}}]
\label{prop:relativeInt}
Let $A \in \ZZ^{d \times m}$ be such that $\ones \in \rsp(A)$ and let $\mathcal{M}_A$ be the corresponding log-linear model. Suppose we observe a vector of counts $u \in \ZZ^m_{\geq 0}$. Then the MLE given $u$ exists in $\mathcal{M}_A$ if and only if $A \bar{u}$  lies in the relative interior of
the polytope $P(A)$.
\end{proposition}

In particular, we see that, indeed, if $u$ has all entries positive, the MLE always exists. However, if $u$ has some entries zero, the MLE may or may not exist.

Following \cite[Section 4.2.3]{Lauritzen}, we define the \textit{extended log-linear model} $\overline{\mathcal{M}_A}$  to be the closure of $\mathcal{M}_A$ in the Euclidean topology on $\RR^m$. The extended model allows for distributions that have some zero coordinates. 
The MLE always exists for the extended model, because it is compact and the likelihood function is continuous.
If the MLE given $u$ does not exist in $\mathcal{M}_A$, we refer to the MLE given $u$ in the extended model $\overline{\mathcal{M}_A}$ as the \emph{extended MLE} given $u$. 

Since the likelihood function \eqref{eq:likelihoodiscrete} is strictly concave for log-linear models, the MLE is unique if it exists, and similarly for the extended MLE.

\subsection{Relating stability to the MLE}
\label{sec:stability_MLE}

We now describe connections between existence of the MLE for a log-linear model and stability under a torus action.

\begin{theorem}
\label{thm:MLEpolystableTorus}
Consider a vector of counts $u \in \ZZ^{m}_{\geq 0}$ with sample size $u_+ = n$, matrix $A \in \ZZ^{d \times m}$ with $\ones \in \CC^m$ in the rowspan, and vector $b = Au \in \ZZ^d$. 
The stability under the action of the complex torus $\GT_d$ given by matrix $n A$  with linearization $b$ is related to ML estimation in $\Mcal_A$ as follows.
\[\begin{matrix}
(a) & \ones \text{ unstable} &  & \text{does not happen} \\
(b) & \ones \text{ semistable} & \Leftrightarrow & \text{extended MLE exists and is unique} \\
(c) & \ones \text{ polystable} &  \Leftrightarrow & \text{MLE exists and is unique} \\
(d) & \ones \text{ stable} &  & \text{does not happen}
\end{matrix}
\]
\end{theorem}

\begin{proof}
We refer to the conditions for the different notions of stability, coming from the Hilbert-Mumford criterion in Theorem~\ref{thm:HMtorus}.
By Proposition~\ref{prop:relativeInt}, the MLE of $u$ exists if and only if 
$b$ lies in the relative interior of the polytope $P(nA)$, which is the condition for polystability in Theorem~\ref{thm:HMtorus}.

It remains to see that the cases of unstable and stable do not occur. The all-ones vector $\ones$ can never be unstable with respect to the action in Theorem~\ref{thm:MLEpolystableTorus}, because $b = Au$ is in the polytope $P(nA)$.
Finally, the stable case also cannot arise, due to the assumption that the vector $\ones$ lies in the row span of $A$, as follows. 
Writing $\ones$ as a linear combination of the rows, i.e. $r\T A = \ones$,
we have that all vectors $a_j$ lie on the hyperplane $r_1 x_1 + \cdots + r_d x_d = 1$ and the polytope $P(A)$ has empty interior in $\RR^d$.
\end{proof}

We remark that we could take any other vector of full support in Theorem~\ref{thm:MLEpolystableTorus}. 
The theorem shows that MLE existence can be tested by checking polystability under the group action. We now give alternative characterizations that involve semistability, which has advantages over polystability. The semistability of $v$ can be checked by evaluating generators of the null cone at $v$. If all generators vanish then $v$ is unstable, otherwise it is semistable. 
 
 \begin{proposition}
\label{prop:intersectNullCones}
For a vector of counts $u \in \ZZ^m_{\geq 0}$ with $u_+ = n$  and $A \in \ZZ^{d \times m}$, the MLE given $u$ exists if and only if there is some $b \in \ZZ^d$, of the form $b = Av$ for $v \in \RR_{>0}^m$, such that $u$ is semistable for the torus action given by matrix $n A$ with linearization $b$. 
\end{proposition}

\begin{proof}
We first assume that the MLE given $u$ exists. Since the vector $Au$ lies in the polytope $P_u(n A)$, the vector $u \in \ZZ_{\geq 0}^m$ is semistable for the action given by matrix $n A$ with linearization $Au$, by Theorem~\ref{thm:HMtorus}(b). Moreover, since $Au$ is in the relative interior of $P(n A)$, by Proposition~\ref{prop:relativeInt}, the vector $Au$ is of the form $Av$ for some $v \in \RR_{>0}^m$. 

Conversely, if the MLE given $u$ does not exist, then $Au$ lies on the boundary of the polytope $P(n A)$. Hence the whole polytope $P_u(n A)$ is contained in the boundary. 
Thus, for every $b \in 
\ZZ^d$ of the form $b = Av$ for  $v \in \RR_{>0}^m$, we have $b \notin P_u(n A)$. Then $u$ is unstable under the torus action given by matrix $nA$ with linearization~$b$, by Theorem~\ref{thm:HMtorus}(a). \end{proof}

To test MLE existence with Proposition~\ref{prop:intersectNullCones}, we need to test null cone membership for multiple linearizations.
We now discuss a different approach, involving one null cone.
For a vector $b \in P(A)$ we denote by $F_b(A)$ the minimal face of the polytope $P(A)$ that contains~$b$; see Figure~\ref{fig:faces_Fb}.

\begin{proposition}
\label{prop:intersectIrredComp}
Consider a vector of counts $u \in \ZZ^m_{\geq 0}$ with $u_+ = n$ and $A \in \ZZ^{d \times m}$.
The intersection of the irreducible components of the null cone for the torus action given by matrix $n A$ with linearization $b = Au$
is
$\langle e_j \mid  n a_j \notin F_b(nA) \rangle$. 

In particular, the MLE given  $u$ exists in $\Mcal_A$ if and only if 
the intersection of the irreducible components of the null cone
is~$\lbrace 0 \rbrace$.
\end{proposition}

\begin{proof}
Define $A' := n A$ and consider the polytope $P(A')$, the convex hull of $a_j' := n a_j$.
We consider the null cone under the torus action given by matrix $A'$ with linearization $b$. 
A linear space $\langle e_j \mid j \in J \rangle$ is in the null cone if and only if $b \notin P_J(A')$, by Proposition~\ref{prop:null_cone_linear_spaces}.

We will show that $e_j$ is contained in every irreducible component of the null cone if and only if $a_j' \notin F_b(A')$.
From this, the second paragraph of the statement follows because the MLE given $u$ exists if and only if $b = Au$ is in the relative interior of the polytope $P(A')$, i.e. $b$ does not lie on a proper face, and $F_b(A') = P(A')$.

Consider an index $j$ with $a_j' \notin F_b(A')$. 
All possible expressions for $b$ as $b = Av$ for some $v\geq 0$ have $v_j = 0$, since $F_b(A')$ is a face of $P(A')$.
Let $J \subseteq [m]$ be
such that $b \notin P_J(A')$, 
i.e. $\langle e_j \mid j \in J \rangle$ is in the null cone.
Taking $J' = J \cup \{ j \}$, the polytope $P_{J'}(A')$ still does not contain~$b$.
Hence, $e_j$ lies in an irreducible component of the null cone that contains $\langle e_{j'} \mid j' \in J' \rangle$;
so $e_j$ lies in every irreducible component.

Conversely, consider an index $j$ with $a_j' \in F_b(A')$. We show that there exists an irreducible component of the null cone that does not contain $e_j$. For each facet $F \subseteq F_b(A')$, let $v_F$ be a vector with $\supp(v_F) = \{ k \mid a_k' \in F \}$, and take $w_F$ with $\supp(w_F) = \supp(v_F) \cup \{ j \}$. The union of $P_{w_F}(A')$ over facets $F \subseteq F_b(A')$ is the whole polytope $F_b(A')$, so $b \in P_{w_F}(A')$ for some facet~$F$. 
By the minimality of $F_b(A')$, we have $b \notin P_{v_F}(A')$.
Hence $\langle e_k \mid a_k' \in F \rangle$ is contained in an irreducible component of the null cone but, since $b \in P_{w_F}(A')$, the irreducible component does not contain $e_j$. 
\end{proof}

\begin{figure}[htbp]
    \centering
    \includegraphics[width=3cm]{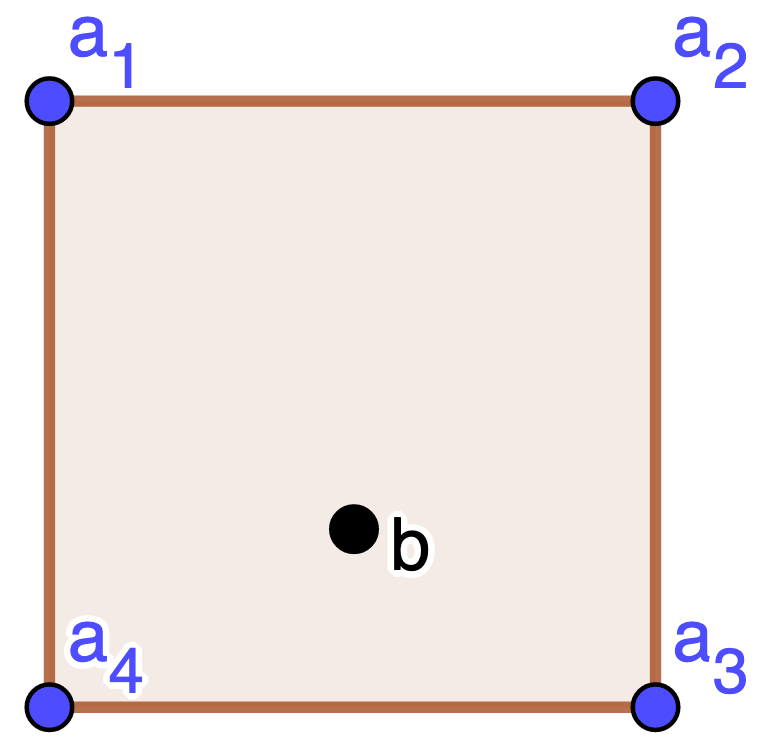} \qquad 
    \includegraphics[width=3cm]{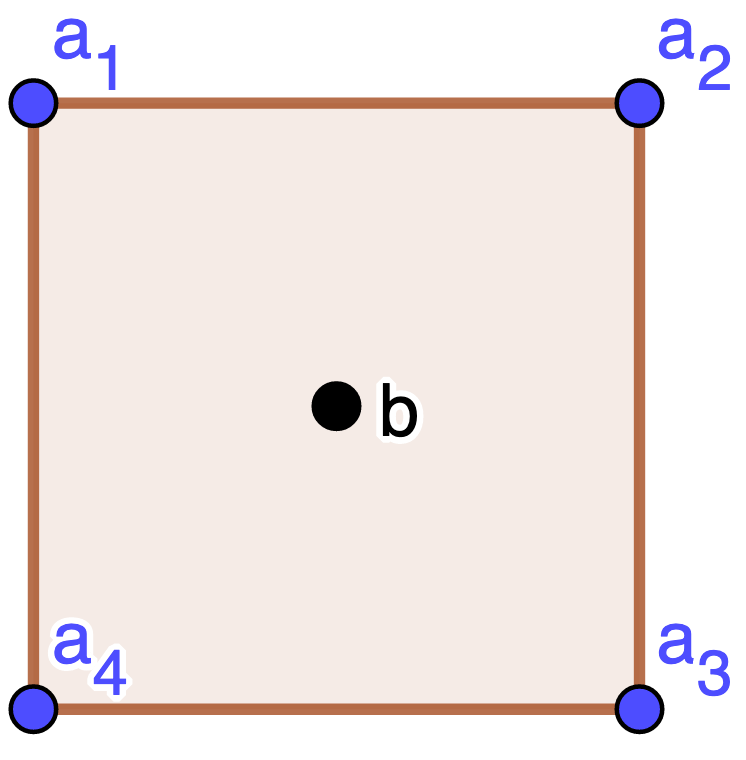}
    \qquad
    \includegraphics[width=3cm]{squares3_2and3.png}
        \qquad
    \includegraphics[width=3cm]{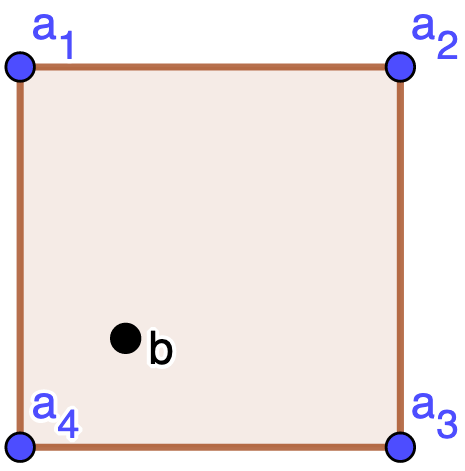}
    \caption{The face $F_b(A)$ of $P(A)$, for four choices of $b \in \ZZ^2$. 
    For example, the leftmost picture displays the face $P_J(A)$ for $J = \{1,2,3,4\}$.
    The vectors $a_i$ outside of the face are in the intersection of all the irreducible components of the null cone, see Proposition~\ref{prop:intersectIrredComp}. For the corresponding components of the null cone, see Figure~\ref{fig:max_not_containing_b}.}
    \label{fig:faces_Fb}
\end{figure}

\begin{example}
We illustrate Proposition~\ref{prop:intersectIrredComp} for the log-linear model $\Mcal_A$, where
\[
A = \begin{blockarray}{ccccccccc}
& p_{000} & p_{001} & p_{010} & p_{011} & p_{100} & p_{101} & p_{110} & p_{111} \\ 
\begin{block}{c(cccccccc)}
p_{00+} & 1 & 1 &   &   &   &   &   &   \\ 
p_{01+} &   &   & 1 & 1  &  &   &   &   \\ 
p_{10+} &   &   &   &  & 1  & 1 &   &   \\ 
p_{11+} &   &   &   &   &   &   & 1 & 1 \\ 
p_{+00} & 1 &   &   &  & 1  &   &   &   \\ 
p_{+01} &   & 1 &   &   &   & 1 &   &   \\ 
p_{+10} &   &   & 1 &   &   &   & 1 &   \\ 
p_{+11} &   &   &   &  1 &  &   &   & 1 \\
\end{block}
\end{blockarray}\]
 This is
 the graphical model on three binary random variables $X_i$ given by the path graph $1$---$2$---$3$,
defined by the conditional independence relation $X_1 \ci X_3 | X_2$. 
To identify the graphical model with $\Mcal_A$, we identify $\RR^8$ with $\RR^{2 \times 2 \times 2}$ and label the columns of $A$ by entries $p_{ijk}$.
The sufficient statistics of the model are the eight marginals $p_{ij+} := p_{ij0} + p_{ij1}$ and $p_{+ij} := p_{0ij} + p_{1ij}$, where $(i,j) \in \{ 0, 1\}^2$.

We compute the irreducible components of the null cone for
the torus action given by matrix $nA$ with linearization $Au$, for
various $u \in \ZZ^8$.
The null cone is 
the zero locus of those monomials in the ring $\CC[x_1, \ldots, x_8]$ such that the supports of their exponent vectors index minimal sub-polytopes of $P(nA)$ that contain $b$, as in Proposition~\ref{prop:vanishing_monomials}.

Let $u = \begin{bmatrix} 1 & 0 & 1 & 0 & 0 & 1 & 0 & 1 \end{bmatrix}\T$. Then $b = \ones \in \RR^8$ and the null cone  is the vanishing locus of
$x_1 x_3 x_6 x_8$, 
$ x_1 x_4 x_6 x_7$, 
$ x_2 x_3 x_5 x_8$, and
$x_2 x_4 x_5 x_7$.
The irreducible components only intersect at $\{ 0 \}$, hence the MLE given $u$ exists in $\Mcal_A$. 

Let $u = \begin{bmatrix} 1 & 1 & 1 & 1 & 0 & 1 & 1 & 0 \end{bmatrix}\T$. Then $b = \begin{bmatrix} 2 & 2 & 1 & 1 & 1 & 2 & 2 & 1\end{bmatrix}$ and the null cone is the vanishing locus of
$x_1 x_2 x_3 x_4 x_6 x_7$,
$x_2^2 x_3 x_4 x_5 x_7$, 
$x_1 x_2 x_3^2 x_6 x_8$, and 
$x_2^2 x_3^2 x_5 x_8$.
The irreducible components only intersect at $\{ 0 \}$, so the MLE given $u$ exists in $\Mcal_A$. 

When $u = \begin{bmatrix} 1 & 0 & 1 & 0 & 0 & 1 & 0 & 0 \end{bmatrix}\T$, the null cone is the vanishing locus of
$x_1 x_3 x_6$ and $x_2 x_3 x_5$. 
The irreducible components intersect at $\langle e_4, e_7, e_8 \rangle$, hence the MLE given $u$ does not exist in $\Mcal_A$. We can also see this from Theorem~\ref{prop:relativeInt}, as follows. The vector $b = Au = \begin{bmatrix} 1 & 1 & 1 & 0 & 1 & 1 & 1 & 0 \end{bmatrix}\T$ has some zero entries. Since $P(A)$ only contains non-negative points, $b$ must lie on the boundary of $P(A)$.\hfill\exSymbol
\end{example}

\subsection{The moment map gives the MLE}
\label{sec:torusmoment}
In Theorem~\ref{thm:MLEpolystableTorus}, we compared two optimization problems: finding the MLE in a log-linear model, and norm minimization in an orbit under a related torus action.
More specifically, we have seen that one problem attains its optimum if and only if the other one does.
We now describe how these two optima are related via the moment map.
For this, we consider two possible closures of the log-linear model $\mathcal{M}_A$.
The Euclidean closure of the model is the extended log-linear model $\overline{\mathcal{M}_A}$.
We also consider the smallest Zariski closed subset of $\Delta_{m-1}$ containing $\mathcal{M}_A$, 
denoted $\overline{\mathcal{M}_A}^Z$, i.e. the Zariski closure of $\mathcal{M}_A$ in $\RR^m$ intersected with the simplex $\Delta_{m-1}$.
Its defining ideal is given in~\eqref{eq:ZarMA}.
In the proof of the following theorem, we use the result that these two closures are equal~\cite[Theorem 3.2]{GeigerMeekSturmfels}.

\begin{theorem}
\label{thm:MLEviaMomentMap}
Let $u \in \ZZ_{\geq 0}^m$ be a vector of counts with $u_+ = n$.
Consider a matrix $A \in \ZZ^{d \times m}$ with $\ones \in \rsp(A)$, and let $b = Au \in \ZZ^d$. Consider the orbit closure of $\ones$ under the torus action of $GT_d$ given by matrix $n A$ with linearization $b$. Let $q \in \CC^m$ be a point in the orbit closure where the moment map vanishes. Then the extended MLE given $u$ for the model $\mathcal{M}_A$ has $j$th entry
\begin{equation}\label{eq:orbitMLE} \frac{|q_j|^2}{\| q\|^2}.\end{equation}
If $\ones$ is polystable, then this vector is the MLE.
\begin{proof}
At a point $q \in \CC^m$ where the moment map vanishes, we have $nAq^{(2)} = \| q \|^2 b$ by Theorem~\ref{thm:kempfNessTorus}. Consider the vector $q'$ with $j$th entry as in~\eqref{eq:orbitMLE}.
We show that $q'$ is the extended MLE given $u$ in $\mathcal{M}_A$.
Since $q' \in \Delta_{m-1}$ and $A q' = A \frac{u}{n}$, it remains to show that $q' \in \overline{\mathcal{M}_A}$. Using the equality $\overline{\mathcal{M}_A}^Z = \overline{\mathcal{M}_A}$, it suffices to show that $q'$ satisfies the equations in~\eqref{eq:ZarMA}.
Since $q$ lies in the orbit closure of $\ones$, 
it satisfies the equations in~\eqref{eq:ZarMA}, where $A$ is replaced by the matrix with $(i,j)$ entry $na_{ij} - b_j$. That is,
it satisfies
$ q^v - q^w = 0$ for all $v,w \in \ZZ_{\geq 0}^m$ with $n A v - b v_+ = n A w - b w_+$. We show that this covers all pairs of vectors $v,w$ with $Av = Aw$. Indeed, if $Av = Aw$ then $v_+ = w_+$, because $\ones$ is in the row span of $A$. Hence $q \in \overline{\mathcal{M}_A}$.
Now we conclude that $q'$ also satisfies the equations in~\eqref{eq:ZarMA}, as follows. 
For each equation $q^v = q^w$, we can take norms on both sides and square both sides. The equality $v_+ = w_+$ then shows that $(q')^v = (q')^w$, since the denominator is present on both sides with equal power $\| q \|^{2 v_+}$.

In the polystable case, the vector $q$ is in the orbit of $\ones$, hence has all entries positive.
Thus the entries of $q'$ are also positive, so $q'$ is the MLE given $u$ in $\mathcal{M}_A$.
\end{proof}
\end{theorem}

\begin{example}
\label{ex:loglinear}
Consider the log-linear model $\mathcal{M}_A$, and vector of counts $u$, where
\[ A = \begin{bmatrix} 2 & 1 & 0 \\ 0 & 1 & 2 \end{bmatrix} , \qquad 
u = \begin{bmatrix} 2 \\ 1 \\ 1 \end{bmatrix}, \qquad b = Au = \begin{bmatrix} 5 \\ 3 \end{bmatrix}.
\]
This model is the plane conic $x_2^2 = x_1x_3$.  The existence of the MLE given $u$ in $\mathcal{M}_A$ can be characterized by the torus action given by  matrix $nA$ with linearization~$b$, by Theorem~\ref{thm:MLEpolystableTorus}, where $n = u_+ = 4$.
Since $b$ is a positive combination of the columns of $A$, the vector $\ones$ is polystable under this action and the MLE given $u$ exists. 
The MLE relates to a point of minimal 2-norm in the orbit of $\ones$ under the torus action given by matrix $n A$ with linearization $b$, by Theorem~\ref{thm:MLEviaMomentMap}. We show how to obtain the MLE from a point $q$ of minimal norm in the orbit of $\ones$.

Since $q$ lies in the orbit of $\ones$, its entries are $q_j = \lambda_1^{n a_{1j} - b_1} \lambda_2^{n a_{2j} - b_2}$, where $\lambda_i$ are non-zero complex numbers, ie.
$ q = \begin{bmatrix} \lambda^3 & \lambda^{-1} & \lambda^{-5} \end{bmatrix}\T$ where $\lambda = \frac{\lambda_1}{\lambda_2}$. 
Moreover, the moment map vanishes at $q$, so we have $n A q^{(2)} = \| q \|^2 b$. Combining these, gives the condition 
$ 3 \nu^2 - \nu - 5 = 0 $, 
where $\nu = |\lambda|^{8}$,
and we obtain
that the MLE is
\[
\phantom{aaaaaaaaaaaaaaaaa}
\hat{p} = \frac{1}{\nu^2 + \nu + 1}\begin{bmatrix} \nu^2 \\ \nu \\ 1 \end{bmatrix} =\begin{bmatrix}\frac{31+\sqrt{61}}{4 \sqrt{61}+52}\cr \frac{3+3 \sqrt{61}}{4 \sqrt{61}+52}\cr \frac{9}{2 \sqrt{61}+26}\end{bmatrix}
\sim  \begin{bmatrix} 0.4662 \\ 0.3175 \\ 0.2162 \end{bmatrix}.
\phantom{aaaaaaaaaaaaa}\diamondsuit
\]
\end{example}

Theorem~\ref{thm:MLEviaMomentMap} shows that the MLE can be obtained from norm minimization on an orbit. It suggests the possibility of using algorithms from invariant theory to compute the MLE, as we describe in Section~\ref{sec:scaling}. In the next example, we motivate the study of these algorithms, returning to the independence model. 

The independence model on a pair of discrete random variables is a log-linear model, as we saw in a special case in Example~\ref{ex:indep}. In the following example, we apply Theorem~\ref{thm:MLEviaMomentMap} to obtain the MLE given $u$ for the independence model from the point of minimal 2-norm in the orbit 
of $\ones \otimes \ones$
under a torus action.

\begin{example}
\label{ex:indep2}
The independence model on a pair of random variables, each with $m$ states, is the log-linear model $\Mcal_A$, where \begin{equation}\label{eqn:Aforindependence}
    A = \begin{bmatrix} & & \\ & I_m \otimes \ones & \\ & &  \\ 
    & & \\ 
& \ones \otimes I_m & \\ & &  \end{bmatrix} \in \ZZ^{2m \times m^2}.
\end{equation} 
The first $m$ rows are $I_m \otimes \ones$ and second $m$ rows are $\ones \otimes I_m$, where $I_m$ is the $m \times m$ identity matrix, and $\ones$ is the all-ones vector of length $m$. 
The Kronecker product $A_1 \otimes A_2$ of two matrices $A_k \in \RR^{m_k \times n_k}$ is a matrix of size $m_1 m_2 \times n_1 n_2$. We index its rows by $(i_1,i_2)$ where $i_k$ ranges from $1$ to $m_k$, and its columns by $(j_1,j_2)$, where $j$ ranges from $1$ to $n_k$. The entry of $A_1 \otimes A_2$ at index $((i_1,i_2),(j_1,j_2))$ is $(A_1)_{i_1 j_1} (A_2)_{i_2 j_2}$. 
See Example~\ref{ex:indep} for~\eqref{eqn:Aforindependence} in the case $m=3$.

The model is the orbit of the all-ones matrix under the left-right action of $\GT_m(\RR) \times \GT_m(\RR)$ on the space of $m \times m$ matrices, after restricting to positive entries that sum to one.
Equivalently, the model is the orbit under the action in~\eqref{eq:torusd} of the torus $\GT_{2m}(\RR)$ on $\CC^{m \times m}$ given by the matrix $A$ in~\eqref{eqn:Aforindependence}, again after restricting to positive entries that sum to one. Equivalently, the model consists of all rank one matrices with positive entries summing to one.

Given a data matrix $u \in \mathbb{Z}^{m \times m}$, we consider the orbit of the all-ones matrix $\ones \otimes \ones \in \CC^{m \times m}$, under the action of $\GT_{2m}(\CC)$ given by the matrix $nA$ with linearization $b$, where 
$A$ is~\eqref{eqn:Aforindependence}, the sample size is $n = u_{++}$, and $b = Au \in \ZZ^{2m}$. 
We seek a matrix in the orbit closure of $\ones \otimes \ones$ at which the infimum norm is attained.
By Kempf-Ness, such matrices are those at which the moment map vanishes.
The vanishing of the moment map at $q \in \CC^{m \times m}$ gives, by Theorem~\ref{thm:kempfNessTorus},
    \[ n \begin{bmatrix} |q_{1+}|^2 \\ \vdots \\ |q_{m+}|^2 \\ |q_{+1}|^2 \\ \vdots \\ |q_{+m}|^2 \end{bmatrix} = \| q \|^2 \begin{bmatrix} u_{1+} \\ \vdots \\ u_{m+} \\ u_{+1} \\ \vdots \\ u_{+m} \end{bmatrix}.\]
We relate the MLE to the matrix $p$ with entries $p_{ij} = \frac{|q_{ij}|^{2}}{\| q\|^2}$.
The matrix $p$ has non-negative entries summing to one, and the same row and column sums as the empirical distribution $\bar{u}$. It remains to show that $p$ is in $\Mcal_A$. 
For $g = \begin{bmatrix}\lambda_1 & \cdots & \lambda_m & \mu_1 & \cdots & \mu_m \end{bmatrix}$ in $\GT_{2m}(\CC)$, we have
    \[ \left( g \cdot (\ones \otimes \ones) \right)_{ij} = \lambda_i^{n} \mu_j^{n} \left( \prod_{k = 1}^m \lambda_k^{-b_k} \right) \left( \prod_{l = 1}^{m} \mu_l^{-b_{m+l}} \right) .\]
Hence $p$ is a scalar multiple of the matrix with $(i,j)$ entry
$|\lambda_i|^{2n} |\mu_j|^{2n}$,
and all such matrices have rank at most one.
The latter is a closed condition, so any non-zero $p$ obtained from the orbit closure of $\ones \otimes \ones$ has rank one.
Hence $p$ lies in the closure of the independence model. If the orbit is closed, all entries are positive and it is the MLE. 
Otherwise, it is the extended MLE; see Threorem~\ref{thm:MLEpolystableTorus}.\hfill\exSymbol
\end{example}

\section{Scaling algorithms}
\label{sec:scaling}

We saw in Theorem~\ref{thm:MLEviaMomentMap} that the MLE in a log-linear model can be obtained from a point of minimal norm in an orbit.
This connects two problems:
\begin{enumerate}
    \item norm minimization in a complex torus orbit
    \item maximum likelihood estimation in a log-linear model.
\end{enumerate}

Algorithms exist for both problems: the former can be approached with convex optimization methods, and the latter with an algorithm called iterative proportional scaling.  In fact, both families of algorithms  can be thought of as generalizations of two sides of a classical scaling algorithm due to Sinkhorn~\cite{sinkhorn1964}. 
We explain these different generalizations, and how Theorem~\ref{thm:MLEviaMomentMap} completes the circle of algorithms, see Figure~\ref{fig:algorithms}.

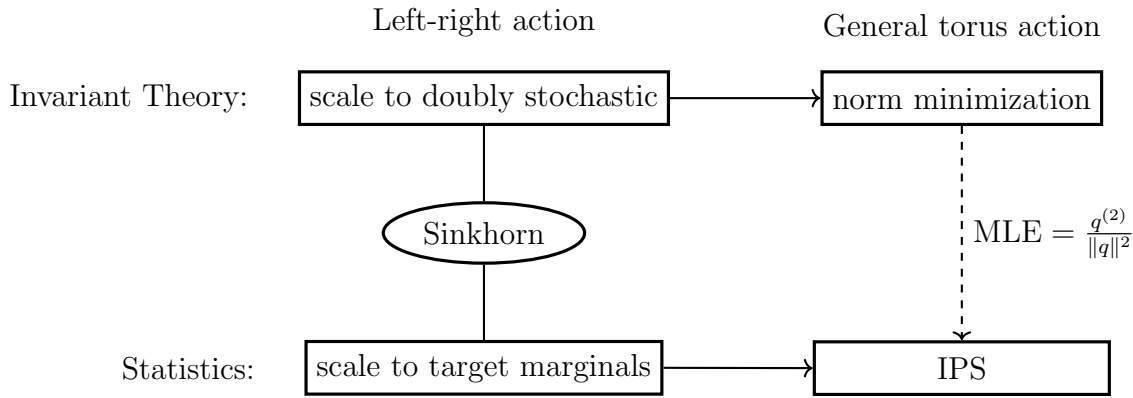
\begin{figure}[htbp]
\centering
\begin{tikzpicture}[
roundnode/.style={ellipse, draw=black, very thick, minimum size=7mm},
squarednode/.style={rectangle, draw=black, very thick, minimum size=7mm},
description/.style={rectangle, very thick, minimum size=5mm},
]
\node[roundnode] (sinkhorn){Sinkhorn};
\node[squarednode] (doubly)[above=of sinkhorn] {scale to doubly stochastic};
\node[squarednode] (target) [below=of sinkhorn] {scale to target marginals};
\node[squarednode] (norm) [right=2cm of doubly] {norm minimization};
\node[squarednode] (ips) [right=2cm of target, below=2.85cm of norm] {$\qquad \quad \;$ IPS $\qquad \quad \;$};
\node[description] (left) [above=0.3cm of doubly] {Left-right action};
\node[description] (torus) [above=0.3cm of norm] {General torus action};
\node[description] (inv) [left=0.5cm of doubly] {Invariant Theory:};
\node[description] (stat) [left=0.5cm of target] {Statistics:};

\draw[thick] (sinkhorn.north) -- (doubly.south);
\draw[thick] (sinkhorn.south) -- (target.north);
\draw[->, thick] (doubly.east) -- (norm.west);
\draw[->, thick] (target.east) -- (ips.west);
\draw[->, dashed, thick] (norm.south) -- node[anchor=west, align=center] {$\text{MLE} = \frac{q^{(2)}}{\|q\|^2}$} (ips.north) ;
\end{tikzpicture}

\caption{Overview of different scaling algorithms.  The historical progression is from left to right, starting with two different sides to Sinkhorn scaling.}
    \label{fig:algorithms}
\end{figure}

\subsection{Sinkhorn scaling}

The classical scaling algorithm of Sinkhorn~\cite{sinkhorn1964} scales a square matrix $M$ with positive entries to a doubly stochastic matrix. That is, one finds diagonal matrices $D_1$ and $D_2$  such that $D_1 M D_2$ has all row sums and all column sums equal to one. The doubly stochastic matrix is obtained by alternatingly scaling the row and column marginals to one. 
A natural extension is to scale the matrix $M$ to other fixed row sums and column sums~\cite{sinkhorn1967}.
Both versions of Sinkhorn scaling are depicted on the left of Figure~\ref{fig:algorithms}.
These algorithms involve the left-right action of a pair of tori $\GT_{m_1} \times \GT_{m_2}$ on an $m_1 \times m_2$ matrix: the algorithms iterate between updates via the left torus and via the right torus. 

Alternating scaling of the rows and columns of a matrix to fixed marginals is an instance of a scaling algorithm which, in the statistics literature, goes back to Deming and Stephan in \cite{IPForiginal}. For the independence model on two variables, the algorithm finds the MLE by alternating between scaling the row sums and the column sums to match the marginals of the empirical distribution.
Given an observed matrix of counts $u \in \ZZ_{\geq 0}^{m \times m}$ with sample size $u_{++} = n$, and
initialized at the uniform distribution, the algorithm has two steps. The $(i,j)$ entry for the two steps is:
\begin{equation}
    \label{eqn:IPS_two_steps}
    \frac{1}{m^2} \, \mapsto \,  \frac{1}{m} \cdot \frac{u_{i+}}{n} \, \mapsto \,   \frac{u_{i+}}{n} \cdot \frac{u_{+j}}{n}.
\end{equation} 
If all entries are positive, the output is the MLE to the independence model given $u$, otherwise it is the extended MLE. 
This is the first example of {\em iterative proportional scaling} (IPS), which we describe in the next subsection.

\subsection{Iterative proportional scaling}

In the previous section, we saw that alternating scaling of a matrix to fixed row and column sums gives the MLE to the independence model, when initialized at the uniform distribution.
This is scaling under a product of tori $\GT_m \times \GT_m$.
We saw in Examples~\ref{ex:indep} and~\ref{ex:indep2} how the independence model fits into the framework of log-linear models. In terms of the group action, this replaces the left-right action of a pair of tori $\GT_m \times \GT_m$ with the action of a single torus $\GT_{2m}$, acting via~\eqref{eq:torusd}, where $A$ is the matrix in~\eqref{eqn:Aforindependence}.

In this section, we explain how Sinkhorn scaling extends to algorithms for maximum likelihood estimation for a general log-linear model, the bottom arrow of Figure~\ref{fig:algorithms}. 

Alternating between matching row and column sums can be extended to hierarchical models, which summarize data by contingency tables \cite{fienberg1970}, by iteratively updating the various marginals.
The approach was extended to more general log-linear models by 
Darroch and Ratcliff in~\cite{IPS-DR}.

For the log-linear model $\Mcal_A$, the MLE $\hat{p}$ must satisfy the equation $A\hat{p} = A \bar{u}$ from Birch's theorem, where $\bar{u} = \frac{u}{n}$ is the empirical distribution. 
IPS finds the extended MLE in $\mathcal{M}_A$ given an empirical distribution $\bar{u} \in \Delta_{m-1}$.
We define IPS for a log-linear model given by a matrix $A \in \ZZ_{\geq 0}^{d \times m}$ whose column sums are all equal. 
 Starting at the uniform distribution $p^{(0)}= \frac{1}{m} \ones$, we iterate until the $k$th update $p^{(k)}$ has sufficient statistics $b^{(k)} = A p^{(k)}$ close to the target sufficient statistics $b = A \bar{u}$, i.e. until \eqref{eq:Birch} holds approximately.
The update step is:
\begin{equation}
    \label{eq:ipsupdate}
p^{(k+1)}_j = \prod_{i = 1}^d \left( \frac{(A\bar{u})_i}{(Ap^{(k)})_i} \right)^{\nicefrac{a_{ij}}{\alpha}} p^{(k)}_j,
\end{equation}
where $\alpha$ is the common column sum of $A$; see \cite[Algorithm 7.3.11]{Sullivant}.
This is the action of a torus element (obtained by componentwise division of $A \bar{u}$ by $A p^{(k)}$
and then componentwise exponentiation by $\nicefrac{1}{\alpha}$)
on the vector $p^{(k)}$.
Here the torus action is given by the matrix $ A$
with linearization $b=0$, see~\eqref{eq:torusAction}.

We can view maximum likelihood estimation as a norm minimization problem in a different way to Theorem~\ref{thm:MLEviaMomentMap}, by interpreting IPS as minimizing KL divergence.

\begin{proposition}
\label{prop:cap_KL}
Consider the log-linear model $\Mcal_A$ where $A \in \ZZ^{d \times m}$ has $\ones$ in its row span. Then there exists a matrix $\tilde{A} \in \mathbb{Q}_{\geq 0}^{(d+1) \times m}$, with all column sums equal, such that  $\mathcal{M}_A = \mathcal{M}_{\tilde{A}}$, iterative proportional scaling in~\eqref{eq:ipsupdate} with matrix $\tilde A$ converges, and at each update step the KL divergence to the MLE decreases.
\end{proposition}

\begin{proof}
The proof of convergence of IPS is given in~\cite[Theorem~1]{IPS-DR}
in the case where the entries of $A$ are real and non-negative with each  column of $A$ summing to one.
There, the authors show that each step of IPS decreases the KL divergence $\KL(\hat{p} \| p^{(k)})$ from the $k$th iterate $p^{(k)}$ to the MLE $\hat{p}$.  Since replacing $A$ by $\frac{1}{\alpha}A$ does not change the update step \eqref{eq:ipsupdate}, the KL divergence also decreases for any matrix  with real and non-negative entries and all column sums equal.

We explain how this covers log-linear models defined by integer matrices with $\ones$ in the row span.
We modify $A$ without changing its row span, i.e. without changing the model $\mathcal{M}_A$.
First, we add a sufficiently large positive integer to every entry of $A$. For a general choice of integer, this does not change $\rsp(A)$ since it adds a multiple of the vector $\mathbbm{1}$, which belongs to $\rsp(A)$, to every row. 
 Second, let $\alpha$ be the maximum of the column sums $a_{+j}$. 
 Add another row to the matrix, with entries $\alpha - a_{+j}$. 
The extra row is a linear combination of $\mathbbm{1}$ and the rows of $A$, so the augmented matrix has the same row span as $A$. The column sums of the augmented matrix $\tilde{A}$ are all $\alpha$. 
\end{proof}

\begin{remark}
We saw in Section~\ref{sec:MLE} that $\hat{p} = {\rm argmin}_{p \in \Mcal} \KL ( \bar{u} \| p )$. Here, we use KL divergence differently, measuring the KL divergence from iterate $p^{(k)}$ to the MLE, $\KL( \hat{p} \| p^{(k)} )$.
\end{remark}

Curiously, when  IPS for log-linear models in~\eqref{eq:ipsupdate} is applied to the independence model, we do not recover the classical IPS with Sinkhorn updates, because the column sums of the integer matrix $A$ for the independence model in \eqref{eqn:Aforindependence} are $\alpha = 2$, hence there is a square root in the update step. 
If, instead, we did IPS with the same matrix $A$ but $\alpha=1$ in~\eqref{eq:ipsupdate} we would recover the two steps in~\eqref{eqn:IPS_two_steps} in a single step. 
This leads naturally to the question of which exponents $\alpha$ achieve convergence, and how the choice of $\alpha$ affects the convergence rate. This is the essence of an open problem in algebraic statistics, see \cite[Section~7.3]{owl}.

\subsection{Norm minimization}

In this section, we explain how Sinkhorn scaling generalizes to norm minimization  in invariant theory;
see the top arrow of Figure~\ref{fig:algorithms}.

The condition that a matrix can be scaled to a doubly stochastic matrix is dual to testing membership in the null cone under a group action, as follows.
We consider pairs of diagonal matrices $(D_1,D_2)$ of determinant one that act on square matrices $M$ via $M \mapsto D_1 M D_2$.
A matrix  does not lie in the null cone under this action if and only if its orbit closure contains a  matrix $M$ such that the matrix with $(i,j)$ entry $|m_{ij}|^2$ is a non-zero scalar multiple of a doubly stochastic matrix~\cite[Corollary 3.6]{GargOliveira}.
This is an instance of Kempf-Ness.
Norm minimization on the orbit of a square matrix either converges to zero, or to a non-zero matrix $M$ at which the moment map vanishes.
The condition $\mu(M) = 0$ translates to the matrix with entries $|m_{ij}|^2$ being a scalar multiple of a doubly stochastic matrix.
So we see that norm minimization scales to a non-zero multiple of a doubly stochastic matrix, if such a matrix exists in the orbit.

Norm minimization on an orbit can be considered for a wide range of groups and their actions.
If a group $G$ can be expressed as a product of groups, then the alternating minimization idea from Sinkhorn's algorithm generalizes. An important example of this is operator scaling, which solves the scaling problem for the left-right action of $\SL_{m_1}(\CC) \times \SL_{m_2}(\CC)$ on the space of matrix tuples $(\CC^{m_1 \times m_2})^n$.
We discuss connections between operator scaling and statistics in our companion paper~\cite{Gaussianpaper}.

We consider the norm minimization problem for the action of the torus $\GT_d(\CC)$ given by the matrix $n A$ with linearization $b$.
We take $b = Au$, where $u$ is a vector of counts.
By Kempf-Ness, Theorem~\ref{thm:kempfNessTorus}, a vector is not in the null cone if and only if there is a non-zero vector $q$ in its orbit closure satisfying $Aq^{(2)} = \Vert q \Vert^2b$. Hence, the problem of scaling a vector by acting with the torus to such a non-zero vector $q$ is dual to testing membership in the null cone under the torus action. This duality generalizes the discussion of doubly stochastic matrices above.

Since the vector $\ones$ is semistable, see Theorem~\ref{thm:MLEpolystableTorus}, norm minimization converges to such a non-zero vector $q$. This is a convex optimization problem, as follows.
Consider the action of $\GT_d(\CC)$
given by matrix $A' = nA - b \otimes \ones \in \ZZ^{d \times m}$. 
For a torus element $(\lambda_1, \ldots, \lambda_d)$, the coordinate change $y_i := \log \vert \lambda_i \vert^2$ gives
    \[
        \capac(\ones) = \inf_{\lambda \in \GT_d(\CC)} \, \| \lambda \cdot \ones \|^2 
        =  \inf_{\lambda \in \GT_d(\CC)} \, \sum_{j=1}^m \prod_{i=1}^d \vert \lambda_i \vert^{2 {a'_{ij}}}
        = \inf_{y \in \RR^d} \, \sum_{j=1}^m \exp \langle y, a'_{j} \rangle.
    \]
Convexity then follows from the fact that each exponential function is convex and a sum of convex functions is convex. This minimization problem is known as geometric programming. 
Hence, common algorithms from the vast literature on convex optimization can be used to compute the capacity and find the MLE, e.g. interior point methods \cite{buergisser2020interiorpoint} or ellipsoid methods.

\subsection{Comparison of algorithms}

We have seen in the previous two subsections that IPS and norm minimization are generalizations of Sinkhorn scaling that have emerged in different communities. 
Theorem~\ref{thm:MLEviaMomentMap} closes the cycle of algorithms from different communities, by showing how to obtain the (extended) MLE from a complex point of minimal norm in an orbit (or orbit closure); see Figure~\ref{fig:algorithms}.

This bridges several differences between IPS and norm minimization.
We summarize these differences here.
First, when computing the capacity, the norm is minimized along a \emph{complex} orbit closure (see Theorem~\ref{thm:MLEviaMomentMap}), whereas every step in IPS involves \emph{real} numbers.
Secondly, the torus action given by  matrix $n A$ that is used for computing the capacity is linearized by  $b = Au$ (see Theorem~\ref{thm:MLEviaMomentMap}), whereas IPS uses the action given by matrix $ A$ with trivial linearization $b=0$.
Finally,
the objective functions differ: the capacity is defined in terms of the Euclidean norm, which does not appear in IPS; instead 
IPS minimizes KL divergence (see Proposition~\ref{prop:cap_KL}). In the following example we see that, while IPS decreases the KL divergence to the MLE, it may increase the Euclidean norm.

\begin{example}
Consider the matrix $A$ and vector of counts $u$ from Example~\ref{ex:loglinear}.
We can use IPS to compute the MLE in $\Mcal_A$.
We start at the uniform distribution $p^{(0)} = \frac13 \ones$ and do update steps as in~\eqref{eq:ipsupdate} with matrix $A$. These IPS steps converge by Proposition~\ref{prop:cap_KL}, since the matrix $A$ has real non-negative entries and all column sums equal. 
We obtain 
$ p^{(1)} = \begin{bmatrix} \frac5{12} & \frac{\sqrt{15}}{12} & \frac14 \end{bmatrix}\T $. Note that the sum of the entries of $p^{(1)}$ is strictly less than one.
The KL divergence from the uniform distribution to the MLE is $KL(\hat{p} \| p^{(0)}) \sim 0.047$, and after the first update it is $KL(\hat{p} \| p^{(1)}) \sim 0.016$. 
However, we have $\| p^{(1)} \|^2 = \frac{49}{144}$, which exceeds $\| p^{(0)} \|^2 = \frac13$.
\hfill\exSymbol
\end{example}

\section{Comparison with multivariate Gaussian models}
\label{sec:comparison}

We highlight similarities and differences with the multivariate Gaussian setting studied in~\cite{Gaussianpaper}. For this, we compare results from this paper with the related results in \cite{Gaussianpaper}.
We start by comparing the two statistical settings.

In the discrete setting, a model is given as a subset of the $(m-1)$-dimensional probability simplex $\Delta_{m-1} \subseteq \RR^m$. In comparison, in the multivariate Gaussian setting, a model is given by a set of concentration matrices $\Psi$ in the cone of positive definite matrices.
For a discrete model $\mathcal{M} \subseteq \Delta_{m-1}$ the data is a vector of counts $u \in \ZZ^m_{\geq 0}$ with $u_+ = n$ the total numbers of observations. The log-likelihood given $u$ at $p \in \mathcal{M}$ is $\sum_{j=1}^m u_j \log(p_j)$. In comparison, for a Gaussian model the data is summarized by the sample covariance matrix $S_Y = \frac{1}{n} \sum_{i=1}^n Y_i Y_i\T$ and the log-likelihood  given a tuple of samples $Y \in (\RR^m)^n$ is
$\log \det (\Psi) - \mathrm{tr} (\Psi S_Y)$.

\subsection{Stability}
In both papers we link notions of stability under group actions to maximum likelihood estimation of certain statistical models: for log-linear models in Theorem~\ref{thm:MLEpolystableTorus} and for Gaussian group models in \cite[Section 3]{Gaussianpaper}. For log-linear models it is enough to consider actions of complex tori on $\CC^m$. In contrast, in \cite{Gaussianpaper} we work with actions of (reductive algebraic) groups over $\RR$ or $\CC$, depending on whether we consider multivariate Gaussian distributions on $\RR^m$ or $\CC^m$. In the log-linear case we study stability of the all-ones vector, while in \cite{Gaussianpaper} we consider the stability notions for the observed tuple of samples. 

For log-linear models, the log-likelihood is always bounded from above and the all-ones vector cannot be unstable. In contrast, in the Gaussian setting a tuple of samples is unstable if and only if the log-likelihood is not bounded from above. 
In both cases, semistability is equivalent to the log-likelihood being bounded from above and polystability is equivalent to the existence of an MLE. In the log-linear case, the MLE is unique if it exists, while for Gaussian group models there may be infinitely many.
In fact, the existence of a unique MLE for Gaussian group models relates to stability of a tuple of samples.
In contrast, for log-linear models the all-ones vector is never stable.

\subsection{MLE computation}
An important similarity between the log-linear and Gaussian settings is that norm minimizers under the respective group actions give an MLE (if it exists), see Theorem~\ref{thm:MLEviaMomentMap} and \cite[Section 3]{Gaussianpaper}. 
For log-linear models, we compute real MLEs from complex torus orbits.
However, for Gaussian group models, we  compute the MLE over $\mathbb{K} \in \{ \mathbb{R}, \mathbb{C} \}$ from orbits over the same field $\mathbb{K}$.
If the all-ones vector is semistable but not polystable, 
Theorem~\ref{thm:MLEviaMomentMap} yields the extended MLE.
However, in the Gaussian case, if a tuple of samples $Y$ is semistable but not polystable there is no corresponding notion of extended MLE.

\subsection{Scaling } From the point of view of scaling algorithms, Sinkhorn's algorithm is  a common origin to both the log-linear and the Gaussian settings. As we described in Section~\ref{sec:scaling}, Sinkhorn scaling to target marginals is IPS for the independence model and this extends to IPS for a general log-linear model. On the Gaussian side, Sinkhorn scaling generalizes to alternating minimization procedures for computing MLEs of matrix normal models  and tensor normal models .
This algorithm is used both in invariant theory for norm minimization and in statistics to compute the MLE; see~\cite{Gaussianpaper}. In contrast, for log-linear models the algorithms from invariant theory and statistics are not the same; see Figure~\ref{fig:algorithms}.

\medskip

\noindent We conclude this paper by pointing out that log-linear models and the Gaussian group models in \cite{Gaussianpaper} are examples of exponential transformation families. Hence, it is an interesting and natural question to ask whether there is a unifying concept that links invariant theory to maximum likelihood estimation for exponential families.

\bigskip

{\small
\paragraph{\textbf{Acknowledgements}}
We are grateful to 
Peter B\"urgisser, Mathias Drton, and Bernd Sturmfels for fruitful discussions.
We are also grateful to an anonymous referee for helpful comments that improved the paper.
CA was partially supported by the Deutsche Forschungsgemeinschaft (DFG) in the context of the Emmy Noether junior research group KR 4512/1-1.
KK was
partially supported by the Knut and Alice Wallenberg Foundation within their WASP
(Wallenberg AI, Autonomous Systems and Software Program) AI/Math initiative. 
Research of PR is funded by the European Research Council (ERC) under the European’s Horizon 2020 research and innovation programme (grant agreement no. 787840). 
}

 \bibliographystyle{alpha}
 \bibliography{literatur}

\begin{thebibliography}{BLNW20}

\bibitem[AKRS21]{Gaussianpaper}
Carlos Am{\'e}ndola, Kathl{\'e}n Kohn, Philipp Reichenbach, and Anna Seigal.
\newblock Invariant theory and scaling algorithms for maximum likelihood
  estimation.
\newblock {\em SIAM Journal on Applied Algebra and Geometry}, 5(2):304--337,
  2021.

\bibitem[Ati82]{AtiyahConvexity}
M.~F. Atiyah.
\newblock Convexity and commuting {H}amiltonians.
\newblock {\em Bull. London Math. Soc.}, 14(1):1--15, 1982.

\bibitem[BFH07]{bishop2007discrete}
Yvonne~M Bishop, Stephen~E Fienberg, and Paul~W Holland.
\newblock {\em Discrete multivariate analysis: theory and practice}.
\newblock Springer Science \& Business Media, 2007.

\bibitem[Bir63]{birch1963}
M.W. Birch.
\newblock Maximum likelihood in three-way contingency tables.
\newblock {\em Journal of the Royal Statistical Society: Series B
  (Methodological)}, 25(1):220--233, 1963.

\bibitem[Bir71]{Birkes}
D.~Birkes.
\newblock Orbits of linear algebraic groups.
\newblock {\em Annals of Mathematics. Second Series}, 93:459--475, 1971.

\bibitem[BLNW20]{buergisser2020interiorpoint}
P.~B\"urgisser, Y.~Li, H.~Nieuwboer, and M.~Walter.
\newblock Interior-point methods for unconstrained geometric programming and
  scaling problems.
\newblock arXiv:2008.12110, 2020.

\bibitem[CLS11]{cox2011toric}
David~A Cox, John~B Little, and Henry~K Schenck.
\newblock {\em Toric varieties}, volume 124.
\newblock American Mathematical Soc., 2011.

\bibitem[Dol03]{Dolgachev}
I.~Dolgachev.
\newblock {\em Lectures on invariant theory}, volume 296 of {\em London
  Mathematical Society Lecture Note Series}.
\newblock Cambridge University Press, Cambridge, 2003.

\bibitem[DR72]{IPS-DR}
J.N. Darroch and D.~Ratcliff.
\newblock Generalized iterative scaling for log-linear models.
\newblock {\em The Annals of Mathematical Statistics}, 43(5):1470--1480, 1972.

\bibitem[DS40]{IPForiginal}
W.E. Deming and F.F. Stephan.
\newblock On a least squares adjustment of a sampled frequency table when the
  expected marginal totals are known.
\newblock {\em The Annals of Mathematical Statistics}, 11(4):427--444, 1940.

\bibitem[DSS09]{owl}
M.~Drton, B.~Sturmfels, and S.~Sullivant.
\newblock {\em {Lectures on Algebraic Statistics}}, volume~39 of {\em
  Oberwolfach Seminars}.
\newblock Birkh\"auser Basel, 2009.

\bibitem[Fie70]{fienberg1970}
S.E. Fienberg.
\newblock An iterative procedure for estimation in contingency tables.
\newblock {\em The Annals of Mathematical Statistics}, 41(3):907--917, 1970.

\bibitem[FR12]{MLEloglinear}
S.E. Fienberg and A.~Rinaldo.
\newblock Maximum likelihood estimation in log-linear models.
\newblock {\em The Annals of Statistics}, 40(2):996--1023, 2012.

\bibitem[GMS06]{GeigerMeekSturmfels}
D.~Geiger, C.~Meek, and B.~Sturmfels.
\newblock On the toric algebra of graphical models.
\newblock {\em The Annals of Statistics}, 34(3):1463--1492, 2006.

\bibitem[GO18]{GargOliveira}
A.~Garg and R.~Oliveira.
\newblock Recent progress on scaling algorithms and applications.
\newblock {\em Bulletin of EATCS}, 2(125), 2018.

\bibitem[GS82]{GuilleminSternberg}
V.~Guillemin and S.~Sternberg.
\newblock Convexity properties of the moment mapping.
\newblock {\em Invent. Math.}, 67(3):491--513, 1982.

\bibitem[Hab74]{haberman1974}
S.J. Haberman.
\newblock Log-linear models for frequency tables derived by indirect
  observation: Maximum likelihood equations.
\newblock {\em The Annals of Statistics}, pages 911--924, 1974.

\bibitem[Hil90]{hilbert1890ueber}
David Hilbert.
\newblock Ueber die theorie der algebraischen formen.
\newblock {\em Mathematische annalen}, 36(4):473--534, 1890.

\bibitem[KN79]{KempfNess}
G.~Kempf and L.~Ness.
\newblock The length of vectors in representation spaces.
\newblock In {\em Algebraic geometry ({P}roc. {S}ummer {M}eeting, {U}niv.
  {C}openhagen, {C}openhagen, 1978)}, volume 732 of {\em Lecture Notes in
  Math.}, pages 233--243. Springer, Berlin, 1979.

\bibitem[Kra84]{KraftBook}
Hanspeter Kraft.
\newblock {\em Geometrische {M}ethoden in der {I}nvariantentheorie}.
\newblock Aspects of Mathematics, D1. Friedr. Vieweg \& Sohn, Braunschweig,
  1984.

\bibitem[Lau96]{Lauritzen}
S.~Lauritzen.
\newblock {\em Graphical models}, volume~17.
\newblock Clarendon Press, 1996.

\bibitem[MFK94]{MumfordGIT}
D.~Mumford, J.~Fogarty, and F.~Kirwan.
\newblock {\em Geometric invariant theory}, volume~34 of {\em Ergebnisse der
  Mathematik und ihrer Grenzgebiete (2) [Results in Mathematics and Related
  Areas (2)]}.
\newblock Springer-Verlag, Berlin, third edition, 1994.

\bibitem[PS05]{ASCB}
L.~Pachter and B.~Sturmfels.
\newblock {\em Algebraic statistics for computational biology}, volume~13.
\newblock Cambridge university press, 2005.

\bibitem[PV94]{PopovVinberg}
Vladimir~L Popov and Ernest~B Vinberg.
\newblock Invariant theory.
\newblock In {\em Algebraic geometry IV}, pages 123--278. Springer, 1994.

\bibitem[Sch86]{SchrijverBook}
Alexander Schrijver.
\newblock {\em Theory of linear and integer programming}.
\newblock Wiley-Interscience Series in Discrete Mathematics. John Wiley \&
  Sons, Ltd., Chichester, 1986.
\newblock A Wiley-Interscience Publication.

\bibitem[Sin64]{sinkhorn1964}
R.~Sinkhorn.
\newblock A relationship between arbitrary positive matrices and doubly
  stochastic matrices.
\newblock {\em The Annals of Mathematical Statistics}, 35(2):876--879, 1964.

\bibitem[SK67]{sinkhorn1967}
R.~Sinkhorn and P.~Knopp.
\newblock Concerning nonnegative matrices and doubly stochastic matrices.
\newblock {\em Pacific Journal of Mathematics}, 21(2):343--348, 1967.

\bibitem[Sul18]{Sullivant}
S.~Sullivant.
\newblock {\em {Algebraic Statistics}}, volume 194 of {\em Graduate Studies in
  Mathematics}.
\newblock AMS, 2018.

\bibitem[Sur00]{Sury}
B~Sury.
\newblock An elementary proof of the {H}ilbert-{M}umford criterion.
\newblock {\em The Electronic Journal of Linear Algebra}, 7, 2000.

\bibitem[Sz{\'e}06]{Szekelyhidi}
G.~Sz{\'e}kelyhidi.
\newblock {\em {Extremal metrics and K-stability}}.
\newblock PhD thesis, Imperial College, University of London, 2006.

\end{thebibliography}

\medskip

\appendix

\section{Hilbert-Mumford for a torus action}
\label{sec:appendixHM} 

Almost all the results in this paper use the polyhedral characterisation of stability under a torus action, given by the Hilbert-Mumford criterion (Theorem~\ref{thm:HMtorus}). In this appendix we present a proof of Theorem~\ref{thm:HMtorus}, in what we hope is an elementary and accessible style.

\begin{remark}[Disregarding the linearization]
\label{rem:no_b}
The setting of Theorem~\ref{thm:HMtorus} is a torus action of $\GT_d$ on $\CC^m$ given by a matrix $A \in \ZZ^{d \times m}$ with linearization $b \in \ZZ^d$. For our statistical connections, it is important to separate the role of $A$ (which is fixed by the model) from that of $b$ (which depends on the data). However, we can remove the need for a linearization by altering the matrix $A$, as follows. The action of $\GT_d$ on $\CC^m$ given by matrix $A \in \ZZ^{d \times m}$ with linearization $b \in \ZZ^d$ is the action given by matrix $A' \in \ZZ^{d \times m}$ with linearization $0 \in \ZZ^d$, where the matrix $A'$ has $j$th column $a_j - b$, see~\eqref{eq:torusAction}. Hence we assume without loss of generality that the linearization is zero for proving Theorems~\ref{thm:HMtorus} and~\ref{thm:kempfNessTorus}.
The effect of the linearization on the moment map is outlined in Section~\ref{sec:the_moment_map}.
\end{remark} 

The classical statement of the Hilbert-Mumford criterion, see~\cite[page~53]{MumfordGIT}, concerns one parameter subgroups. 
For the group $\GT_d$, a one parameter subgroup is given by a map 
\begin{equation}\label{eqn:1ps_sigma}
\begin{split}
\sigma: \CC^\times & \to \GT_d \\ 
\lambda & \mapsto (\lambda^{\sigma_1}, \ldots, \lambda^{\sigma_d} ) ,
\end{split} 
\end{equation}
for some fixed $(\sigma_1,\ldots,\sigma_d) \in \ZZ^d$. We use $\sigma$ to denote both the map and the vector $(\sigma_1,\ldots,\sigma_d)$.
For $v \in \CC^m$, the $j$th entry of $\sigma(\lambda) \cdot v$ is
  \[ 
  (\sigma(\lambda) \cdot v)_j = \lambda^{\langle a_j , \sigma \rangle} v_j . 
  \] 
We consider $\lim_{\lambda \to 0} \sigma(\lambda) \cdot v$. The $j$th entry of the limiting vector is zero for $j \notin \supp(v)$. For $j \in \supp(v)$, we have three possibilities:
    \begin{equation}\label{eq:1psgLimit}
        \left(\lim_{\lambda \to 0} \;\; \sigma(\lambda) \cdot v \right)_j =
        \begin{cases}
            0       & \quad \text{if } \langle \sigma, a_j \rangle > 0\\
            v_j     & \quad \text{if } \langle \sigma, a_j \rangle = 0\\
            \infty  & \quad \text{if } \langle \sigma, a_j \rangle < 0
        \end{cases}
    \end{equation}

The classical statement of the Hilbert Mumford criterion for a torus action is as follows.

\begin{theorem}\label{thm:specialHilbertMumford}
Consider the action of $\GT_d$ on $\CC^m$ via the matrix $A \in \ZZ^{d \times m}$. Given a non-zero $v \in \CC^m$, with zero in its orbit closure, there exists a one-parameter subgroup of $\GT_d$ that scales $v$ to zero in the limit.
\end{theorem}

We give a proof of Theorem~\ref{thm:specialHilbertMumford} following~\cite{Sury}. Other references for the statement of the theorem include~\cite[ Proposition~5.3]{PopovVinberg} and~\cite[Lemma~3.4]{Birkes} for a torus, and~\cite[Theorem 4.1]{Birkes},\cite[Page 53]{MumfordGIT} and \cite[Theorem~5.2]{PopovVinberg} for a general reductive group.

\begin{proof}[Proof of Theorem~\ref{thm:specialHilbertMumford}]
We seek a one parameter subgroup  $\sigma: \CC^\times \to \GT_d$ such that $\lim_{\lambda \to 0} \, \sigma(\lambda) \cdot v$ is zero. 
From the form of a one parameter subgroup from~\eqref{eqn:1ps_sigma} and the limiting behaviour from~\eqref{eq:1psgLimit}, we see that 
this is equivalent to showing that
    \begin{equation}
        \label{eqn:gordan}
    \text{there exists} \, \, \sigma \in \ZZ^d \, \, \text{such that} \, \, \langle \sigma, a_j \rangle > 0 \, \, \text{for all} \, \, j \in \supp(v).
    \end{equation}
    Reordering the entries of $v$, we can assume without loss of generality that $\supp(v) = [k]$ for some $k \leq m$. 
    Then the existence of such a $\sigma \in \ZZ^d$ as in~\eqref{eqn:gordan} is equivalent to the following statement about $A \in \ZZ^{m \times d}$:
    \begin{equation}\label{eq:proofClassicalHM}
\begin{matrix}    \text{if} \, \, t = (t_1,\ldots,t_k) \in \RR^k \backslash \{0\} \, \, \text{is such that} \, \,
        t_1 a_{i1} + \cdots + t_k a_{ik} = 0 \, \, \text{for all} \, \, i \in [d] \\ 
\text{then at least two entries of $t$ are of opposite sign.}
        \end{matrix} 
    \end{equation}
    The equivalence of~\eqref{eqn:gordan} and~\eqref{eq:proofClassicalHM} is~\cite[Lemma~1.1]{Sury}, and is an analogue of Gordan's Theorem~\cite[Section 7.8 Equation~(31)]{SchrijverBook} over the rational numbers.
    Thus it remains to prove~\eqref{eq:proofClassicalHM}.

    Since $v$ has $0$ in its orbit closure, there exists a sequence $\lambda^{(n)} = (\lambda^{(n)}_1,\ldots,\lambda^{(n)}_d) \in \GT_d$ with $\lambda^{(n)} \cdot v \to 0$ as $n \to \infty$. In coordinates, 
    \begin{equation}\label{eq:lambdaN}
        \big( \lambda^{(n)}_1 \big)^{a_{1j}} \cdots \big( \lambda^{(n)}_d \big)^{a_{dj}} \to 0 \quad \text{ as } n \to \infty
        \quad \text{for all} \quad j \in [k] .
    \end{equation}
        The hypothesis of~\eqref{eq:proofClassicalHM} is that we have $t \in \RR^k \backslash \{0\}$ with $t_1 a_{i1} + \cdots + t_k a_{ik} = 0$ for all $i \in [d]$. Without loss of generality, we can assume $t_1$ is non-zero and therefore
\begin{equation}
\label{eqn:ai1}
        -a_{i1} = \frac{t_2}{t_1} a_{i2} + \cdots + \frac{t_k}{t_1} a_{ik} \quad \text{for all} \quad i \in [d] ,
    \end{equation}
    which implies
    \begin{equation}\label{eq:classicalHMcontra}
        \prod_{i=1}^d \Big(\lambda^{(n)}_i \Big)^{-a_{i1}} = 
        \left( \prod_{i=1}^d \left(\lambda^{(n)}_i \right)^{a_{i2}} \right)^{\frac{t_2}{t_1}} \cdots
        \left( \prod_{i=1}^d \left(\lambda^{(n)}_i \right)^{a_{ik}} \right)^{\frac{t_k}{t_1}}.
    \end{equation}
If $\nicefrac{t_j}{t_1} \geq 0$ for all $j \in \{ 2,\ldots,k \}$, then the right-hand side of \eqref{eq:classicalHMcontra} either equals one (if all $\nicefrac{t_j}{t_1}$ are zero) or tends to zero (if there exists some $j$ with $\nicefrac{t_{j}}{t_1} > 0$). But the left-hand side of~\eqref{eq:classicalHMcontra} tends to infinity as $n \to \infty$, since it is the inverse of~\eqref{eq:lambdaN} for $j=1$. Hence $\nicefrac{t_j}{t_1}$ must be strictly negative for some $j$, i.e. $t_1$ and $t_j$ have opposite signs.
\end{proof}
 
Theorem~\ref{thm:specialHilbertMumford} has the following generalization. It can be proved via a polyhedral geometry argument similar to the proof of Theorem~\ref{thm:specialHilbertMumford}.

\begin{theorem}[{\cite[p.~173]{KraftBook}}]\label{thm:generalHilbertMumford}
Consider the action of $\GT_d$ on $\CC^m$ given by matrix $A \in \ZZ^{d \times m}$, and fix $v \in \CC^m$. If $w \in \overline{\GT_d \cdot v} \backslash \GT_d \cdot v$, then there exists a one-parameter subgroup that scales $v$ to an element of $\GT_d \cdot w$ in the limit.
\end{theorem}

\begin{remark} Theorem~\ref{thm:generalHilbertMumford} has an analogue for a general reductive group~$G$, but it only applies with the additional assumption that $G \cdot w$ is the unique closed orbit in $\overline{G \cdot v}$, see \cite[Section~6.8]{PopovVinberg}.
\end{remark} 
  
Equipped with Theorem~\ref{thm:specialHilbertMumford} and Theorem~\ref{thm:generalHilbertMumford}, we now prove Theorem~\ref{thm:HMtorus}.
The proof mostly rests on Theorem~\ref{thm:specialHilbertMumford}; we only use the stronger statement of Theorem~\ref{thm:generalHilbertMumford} for one direction of one of the four cases.
    
\begin{proof}[Proof of Theorem~\ref{thm:HMtorus}]
We first prove parts (a) and (b). If $v=0$, then the polytope $P_v(A)$ is empty, hence $0 \notin P_v(A)$. If $v \neq 0$ is unstable, then there exists some $\sigma \in \ZZ^d$ such that $\langle \sigma, a_j \rangle > 0$ for all $j \in \supp(v)$, by combining Theorem~\ref{thm:specialHilbertMumford} with~\eqref{eq:1psgLimit}. Hence $\sigma$ defines a hyperplane 
    \[ H_{\sigma} = \{ x \in \RR^d \mid \langle \sigma, x \rangle = 0 \} \]
that separates zero from $P_v(A)$. By Farkas' lemma, see \cite[Section~7.3]{SchrijverBook}, such a hyperplane exists if and only if $0 \notin P_v(A)$.

For (c), we first prove that if $0$ is on the boundary of $P_v(A)$, then $v$ is not polystable. We construct a point in the orbit closure of $v$, with support strictly smaller than that of $v$, and hence deduce that the orbit of $v$ is not closed. 
Since $0$ lies on the boundary of $P_v(A)$, it is contained in a minimal face $F \subsetneq P_v(A)$. Since $A$ has integer entries, there is a hyperplane $H_{\sigma}$, with $\sigma \in \ZZ^d$, such that $F = H_{\sigma} \cap P_v(A)$. We choose the sign of $\sigma$ so that it has non-negative inner product with all of $P_v(A)$. This ensures that the limit $w := \lim_{\lambda \to 0} \sigma(\lambda) \cdot v$ exists.
The limit $w$ has $\supp(w) \subsetneq \supp(v)$, since $P_w(A) \subseteq F$. Hence $w \in \overline{\GT_d \cdot v} \backslash \GT_d \cdot v$, and $\GT_d \cdot v$ is not closed.

For the converse direction of (c), we show that if $v$ is semistable but not polystable, then $0 \notin {\rm relint} (P_v(A))$. Let $w' \in \overline{\GT_d \cdot v} \backslash \GT_d \cdot v$. There exists $\sigma \in \ZZ^d$ such that $w := \lim_{\lambda \to 0} \sigma(\lambda) \cdot v \in \GT_d \cdot w'$, by Theorem~\ref{thm:generalHilbertMumford}. We have $\supp(w) \subseteq \supp(v)$ and, moreover, $\supp(w) \subsetneq \supp(v)$ (otherwise $w=v$ by~\eqref{eq:1psgLimit}, a contradiction). 
Hence $\langle \sigma, a_j \rangle > 0$ for all $j \in \supp(v) \backslash \supp(w)$, while $\langle \sigma, a_j \rangle = 0$ for all $j \in \supp(w)$, by~\eqref{eq:1psgLimit}. We obtain $P_v(A) \nsubseteq H_{\sigma}$ and $P_w(A) = H_{\sigma} \cap P_v(A)$, i.e. $P_w(A)$ is a proper face of $P_v(A)$. We have $\GT_d \cdot w = \GT_d \cdot w' \subseteq \overline{\GT_d \cdot v}$ and so $w$ is semistable as $v$ is semistable. By~(b), $0 \in P_w(A)$ and hence $0$ is on the boundary of $P_v(A)$.

It remains to prove (d). We can assume $v$ is polystable, i.e. $0 \in \mathrm{relint}(P_v(A))$. We want to show that the dimension of the stabiliser $\{ \lambda \in \GT_d : \lambda \cdot v = v\}$ is zero if and only if the interior of $P_v(A)$ equals its relative interior (i.e. if and only if $P_v(A)$ is full-dimensional). Since $0 \in P_v(A)$, the equality of the interior and relative interior holds if and only if $U:= \mathrm{span} \{ a_j \colon j \in \supp(v) \}$ equals $\RR^d$. If the stabilizer is positive dimensional, it must contain a one-parameter subgroup, i.e. some $\sigma \in \ZZ^d \backslash \{0\}$ with $\sigma(\lambda) \cdot v = v $ for all $\lambda \in \CC^\times$. Then $\langle \sigma, a_j \rangle = 0$ for all $j \in \supp(v)$, so the orthogonal complement $U^{\perp} \subseteq \RR^d$ contains a line, and $U \neq \RR^d$. Conversely, if $U \neq \RR^d$ then there exists $\sigma \in U^\perp$, which can be chosen to have integer entries, since $A$ has integer entries. The one parameter subgroup $\sigma(\lambda)$ then lies in the stabilizer, which is therefore positive-dimensional.
\end{proof}

\section{Kempf-Ness for a torus action}
\label{sec:appendixKN} 

Many of the results in this paper use the Kempf-Ness theorem for a torus action, as stated in Theorem~\ref{thm:kempfNessTorus}. 
An elementary proof of the Kempf-Ness theorem for torus actions can be found in Sections 1 and 2 of the original paper of Kempf and Ness~\cite{KempfNess}. 
In this appendix, we translate between the setting of~\cite{KempfNess} and our setting, to explain how the results of \cite{KempfNess} give Theorem~\ref{thm:kempfNessTorus}.
In addition, we present an alternative proof of Theorem~\ref{thm:kempfNessTorus}, obtaining it as a consequence of Theorem~\ref{thm:HMtorus}.

As before, we consider the action of $\GT_d$ on $\CC^m$ given by a matrix $A \in \ZZ^{d \times m}$. We can assume without loss of generality that the linearization is $b=0$, by Remark~\ref{rem:no_b}.
We first describe how Theorem~\ref{thm:kempfNessTorus} follows from~\cite{KempfNess}.

\begin{proof}[First proof of Theorem~\ref{thm:kempfNessTorus}]
Let $v \in \CC^m$ and consider the action of $\GT_d$ via the matrix $A \in \ZZ^{d \times m}$. Recall that the moment map at $v$ is the derivative $D_{I} \, \gamma_v$, where 
\[
\gamma_v \colon \GT_d   \to \RR, \qquad 
\lambda  \mapsto \| \lambda \cdot v \|^2.
\]
Identifying the space of $\RR$-linear functionals ${\rm Hom}(\CC^d,\RR)$ with $\CC^d$ gives the moment~map 
    \[ \mu \colon \CC^m \to \CC^d, \qquad v \mapsto 2A v^{(2)},\]
where $v^{(2)}$ is the vector with $j$th coordinate $\vert v_j \vert^2$.

We translate between our notation and that in \cite{KempfNess}. Most importantly, our notion of ``polystable'' is called ``stable'' by Kempf and Ness, see~\cite[page~234]{KempfNess}. Our function $\gamma_v$ is denoted $p_v$ in \cite{KempfNess}. Moreover, ``a critical point $g \in G$ of $p_v$'' in \cite{KempfNess} means the vanishing of the moment map at $g \cdot v$:
    \[ 0 = D_g\, p_v = D_{I}\, p_{g \cdot v} = D_{I}\, \gamma_{g \cdot v} = \mu(g \cdot v).\]
Thus the \emph{polystable} part of Theorem~\ref{thm:kempfNessTorus} is a direct consequence of \cite[Theorem~0.1(a) and Theorem~0.2]{KempfNess} for $G = \GT_d$. The \emph{semistable} part follows from the polystable part, using the fact that any orbit closure for the group $\GT_d$ contains a unique closed orbit, see e.g. \cite[Bemerkung~1 on page~96]{KraftBook}. That unique closed orbit is not the zero orbit if and only if the vector is semistable.
\end{proof}

We end with an alternative proof of Theorem~\ref{thm:kempfNessTorus}, which uses an important connection between the polytope $P_v(A)$ and the moment map $\mu$, see~\cite{AtiyahConvexity} and~\cite{GuilleminSternberg}. The connection relates $P_v(A)$ to the image of the orbit $\GT_d \cdot v$ under the moment map. It is a first example of a moment polytope, an important object of study in invariant theory.

\begin{proof}[Second proof of Theorem~\ref{thm:kempfNessTorus}]
To link the moment map $\mu$ to the polytope $P_v(A)$, we need to rescale $\mu$ as follows:
    \begin{equation}\label{eqn:tildemu}
    \begin{split}
        \tilde{\mu} \colon \CC^m \backslash \{0\} &\to \RR^d,\\
        v &\mapsto \frac{\mu(v)}{2\|v\|^2} = \frac{A v^{(2)}}{\|v\|^2}.
        \end{split}
    \end{equation}
The vector $\nicefrac{v^{(2)}}{\| v \|^2}$ consists of non-negative numbers that sum up to one. Hence $\tilde{\mu}(v)$ is a convex combination of the columns of $A$, therefore $\tilde{\mu}(v) \in P_v(A)$.
The stronger statement
    \begin{equation}\label{eq:PolytopeVsMomentMap}
        \mathrm{relint}\, P_v(A) = \tilde{\mu}(\GT_d \cdot v),
    \end{equation}
was proven independently in \cite[Theorem~2]{AtiyahConvexity} and \cite[Theorem~4]{GuilleminSternberg}, see Remark~\ref{rem:ProjectiveSetting}.
Given~\eqref{eq:PolytopeVsMomentMap}, the polystable case of Theorem~\ref{thm:kempfNessTorus} is a direct consequence of Theorem~\ref{thm:HMtorus}(c), as follows. Polystability is equivalent to $0 \in \mathrm{relint}\, P_v(A)$ which, by~\eqref{eq:PolytopeVsMomentMap}, implies $0 \in \tilde{\mu}(\GT_d \cdot v)$, i.e. that $0 = \frac{A w^{(2)}}{\|w\|^2}$ for some $w \in \GT_d \cdot v$.
As in the first proof, the semistable case of Theorem~\ref{thm:kempfNessTorus} can be deduced from the polystable case, since any orbit closure $\overline{\GT_d \cdot v}$ contains a unique closed orbit.
\end{proof}

\begin{remark}\label{rem:ProjectiveSetting}
The statements in \cite[Theorem 2]{AtiyahConvexity} and \cite[Theorem 4]{GuilleminSternberg} discuss a moment map whose domain is a projective space $\PP_{\CC}^{m-1}$, rather than the space $\CC^m \backslash \{ 0\}$ in~\eqref{eqn:tildemu}.
However, the projective results still allow us to obtain \eqref{eq:PolytopeVsMomentMap}, as follows.
For non-zero $v \in \CC^m$, let $[v]$ be the point in $\PP^{m-1}_{\CC}$ that represents the line $\CC v$. We consider the action of $\GT_d$ on $\PP^{m-1}_{\CC}$ given by the matrix $A \in \ZZ^{d \times m}$. We denote the $\GT_d$-orbit of $[v]$ in $\PP^{m-1}_{\CC}$ by $\GT_d \cdot [v]$ (and denote the orbit closure by $\overline{\GT_d \cdot [v]}$).
The map $\tilde{\mu}$ factors through the projective space $\PP^{m-1}_{\CC}$ via a unique map $\bar{\mu} \colon \PP^{m-1}_{\CC} \to \RR^d$. In fact, $\bar{\mu}$ is the moment map for the action of $\GT_d$ on $\PP^{m-1}_{\CC}$ given by $A \in \ZZ^{d \times m}$. This action fits the setting of \cite{AtiyahConvexity,GuilleminSternberg} because $\PP^{m-1}_{\CC}$ is a compact K\"ahler manifold.
The results \cite[Theorem~2]{AtiyahConvexity} and \cite[Theorem~4]{GuilleminSternberg} give
    \[P_v(A) = \bar{\mu} \left( \overline{\GT_d \cdot [v]} \right).\]
For~\eqref{eq:PolytopeVsMomentMap}, we require a statement for
the orbit of $v$ rather than the orbit closure of $[v]$. The closure $\overline{\GT_d \cdot [v]}$ is the disjoint union of finitely many $\GT_d$ orbits. The orbits relate to $P_v(A)$ as follows.  For each open face $F$ of $P_v(A)$ the set $\bar{\mu}^{-1}(F) \cap \overline{\GT_d \cdot [v]}$ is a single $\GT_d$ orbit in $\PP^{m-1}_{\CC}$, see \cite[Theorem~2]{AtiyahConvexity}. In particular, when $F = \mathrm{relint}\, P_v(A)$ we obtain the orbit $\GT_d \! \cdot \! [v]$. This yields~\eqref{eq:PolytopeVsMomentMap}, since $\bar{\mu}(\GT_d \cdot [v]) = \tilde{\mu}(\GT_d \cdot v)$.
\end{remark}

\bigskip \bigskip

\noindent
\footnotesize {\bf Authors' addresses:}

\smallskip 

\noindent
\noindent Technische Universit\"at M\"unchen, Germany,
\hfill {\tt carlos.amendola@tum.de}

\noindent KTH Royal Institute of Technology, Sweden,
\hfill {\tt kathlen@kth.se}

\noindent  Technische Universit\"at Berlin, Germany,
\hfill {\tt reichenbach@tu-berlin.de}

\noindent  Harvard University, USA,
\hfill {\tt aseigal@math.harvard.edu}

\end{document}